\documentclass[12pt]{article}
\usepackage[utf8]{inputenc}

\usepackage{fetamont}
\usepackage[T1]{fontenc}

\usepackage{amsmath}
\usepackage{amssymb}
\usepackage{amsthm}
\usepackage{mathrsfs}
\usepackage{bm}
\usepackage{bbm}

\usepackage{authblk}

\usepackage{comment}

\usepackage[svgnames,x11names]{xcolor}
\usepackage{hyperref}

\hypersetup{
	colorlinks, 
	linkcolor={SteelBlue},
	citecolor={SteelBlue}, 
	urlcolor={SteelBlue}
}

\usepackage{graphicx}
\usepackage{subfig}
\usepackage{tikz}
\usetikzlibrary{decorations.pathreplacing, tikzmark, calligraphy} 
\usepackage{tikz-cd}

\usetikzlibrary{tikzmark, arrows.meta, calc, bending}

\newtheoremstyle{slanted}
{}
{}
{\slshape}
{}
{\bfseries}
{.}
{ }
{}

\theoremstyle{slanted}
\newtheorem{theorem}{Theorem}[section]
\newtheorem{lemma}[theorem]{Lemma}
\newtheorem{corollary}[theorem]{Corollary}
\newtheorem{proposition}[theorem]{Proposition}
\newtheorem{definition}[theorem]{Definition}
\newtheorem{remark}[theorem]{Remark}

\newtheorem{problem}{Problem}

\DeclareMathAlphabet{\mftm}{T1}{ffm}{m}{n} 
\newcommand{\A}{\mftm{A}} 
\newcommand{\B}{\mftm{B}}
\newcommand{\C}{\mftm{C}}
\newcommand{\D}{\mftm{D}}

\renewcommand{\S}{\mftm{S}}

\newcommand{\ZZ}{\mathbb{Z}}

\newcommand{\RR}{\mathbb{R}}

\newcommand{\NN}{\mathbb{N}}
\newcommand{\M}{\mathscr{M}}
\newcommand{\ind}[1]{\mathbbmss{1}_{#1}}
\newcommand{\tend}[3][]{\xrightarrow[#2\to#3]{#1}}
\newcommand{\occu}[3]{\left|#1_{#2}\right|_{#3}}

\DeclareMathOperator{\type}{type}
\DeclareMathOperator{\Types}{Types}
\DeclareMathOperator{\rec}{rec}

\def\ub#1{
	\leavevmode\hbox{%
		\setbox0\hbox{$#1$}\dp0 0pt
		\vrule height.5ex width.4pt depth.33333ex \kern-.4pt
		\vtop{\hbox{\kern.15em \box0\kern.15em}\kern.33333ex \hrule}%
		\kern-.4pt \vrule height.5ex width.4pt depth.33333ex \kern-.4pt
	}%
}

\def\ubb#1{
	\leavevmode\hbox{%
	\setbox0\hbox{$#1$}\dp0 0pt
	{\color{white}\vrule height.5ex width.4pt depth.33333ex}\kern-.4pt
	\vtop{\hbox{\kern.15em \box0\kern.15em}\kern.33333ex 
		{\color{white}\hrule}}%
	\kern-.4pt {\color{white}\vrule height.5ex width.4pt depth.33333ex}\kern-.4pt
}%
}

\definecolor{Green}{rgb}{0,0.6,0}
\definecolor{Orange}{rgb}{1.0, 0.5, 0.0}

\newcommand\bloc[2][]{\underset{#1}{\ub{#2}}\ }
\newcommand\dessous[2][]{\underset{#1}{\ubb{#2}}\ }

\title{Frequency of patterns in smooth sequences \\ 
over the alphabet $\{1,3\}$}

\author[1]{Damien Jamet}
\author[2]{Irène Marcovici}
\author[3]{L\'eo Poirier}
\author[4]{\\Thierry de la Rue}
\affil[1]{{\small Université de Lorraine, CNRS, LORIA, UMR 7503, F-54000 Nancy, France \protect\\ \texttt{damien.jamet@loria.fr}}}
\affil[2,4]{{\small Univ Rouen Normandie, CNRS, Normandie Univ, LMRS UMR 6085, \protect\\ F-76000 Rouen, France \protect\\ \texttt{irene.marcovici@univ-rouen.fr, thierry.de-la-rue@univ-rouen.fr}}}
\affil[3]{{\small Aix-Marseille Université, CNRS, I2M UMR 7373, F-13000 Marseille, France \protect\\ \texttt{leo.poirier@univ-amu.fr}}}

\date{}

\begin{document}
\bibliographystyle{plainurl}
\sloppy

\maketitle

\begin{abstract}
We provide an ergodic theory framework to study statistical properties of smooth sequences over the odd alphabet $\{1, 3\}$. The arithmetic nature of this alphabet yields a partition of the subshift of smooth sequences based on their local structure, defining a notion of type for those sequences. We describe the substitutive structure of the smaller subshifts obtained by fixing the sequence of types of the successive derivatives of smooth sequences, from which we obtain the unique ergodicity of all these subshifts. A direct consequence is that the asymptotic frequency of any finite pattern in a smooth sequence over $\{1,3\}$ is always well-defined and depends on its type sequence. Finally, we characterize the minimality of these subshifts, and propose some perspectives.
\end{abstract}

\section{Introduction}

The study of combinatorial structures arising from run-length encoding has been a subject of enduring interest in discrete mathematics and theoretical computer science. For many years, the canonical reference for this central object of study has been the 1965 problem posed by Kolakoski \cite{Kol65}, in which he asked for a simple rule describing the self-reading sequence $1 22 11 2 1 22 1 22 11 \dots$ (see Figure \ref{fig::Oldenburger-Kolakoski-Sequence}) and questioned its periodicity. 
\begin{figure}[!h]
	\[ 	\dessous[\Delta(\mathbb{K})]{\mathbb{K}}	\dessous[=]{=}		 \bloc[1]{1} \bloc[2]{22} \bloc[2]{11} \bloc[1]{2} \bloc[1]{1} \bloc[2]{22} \bloc[1]{1} \bloc[2]{22} \bloc[2]{11} \bloc[1]{2} \bloc[2]{11} \bloc[2]{22} \bloc[1]{1} \bloc[1]{2} \bloc[2]{11} \bloc[1]{2} \bloc[1]{1} \bloc[2]{22} \dots 
	\]
	\caption{The Oldenburger-Kolakoski sequence $\mathbb{K}$ and its run-length encoding $\Delta(\mathbb{K}) = \mathbb{K}$.\label{fig::Oldenburger-Kolakoski-Sequence}}
\end{figure}

However, this specific sequence has been introduced and also extensively investigated much earlier by Oldenburger in 1939 \cite{Old39, Brlek2012Kolakoski}. The \textbf{Oldenburger-Kolakoski sequence}, hereafter referred to as such and denoted by $\mathbb{K}$, is the unique sequence over the alphabet $\{1, 2\}$ starting with $1$ 
that is a fixed point of the run-length encoding operator $\Delta$, simply called the \emph{derivative} operator.

While the definition of $\mathbb{K}$ is elementary, its combinatorial and dynamical properties remain largely unknown. In his seminal work on self-generating sequences, Dekking \cite{Dek80} investigated two fundamental structural properties—originally posed as a problem by Kimberling \cite{Kim79}—that remain open to this day:
\begin{itemize}
	\item \textbf{Mirror invariance:} A word is a factor of $\mathbb{K}$ if and only if its mirror image (obtained by exchanging $1$s and $2$s) is also a factor. The same question holds for the reversal invariance (reading from right to left).
	\item \textbf{Recurrence:} Every factor of $\mathbb{K}$ occurs infinitely often. 
\end{itemize}

Regarding the complexity, Dekking \cite{Dek80} conjectured a strictly polynomial growth $p_\mathbb{K}(n) \asymp n^\rho$ with a transcendental exponent $\rho = \frac{\log 3}{\log 3/2} \approx 2.71$. Other fundamental questions remain unanswered, such as Keane's question \cite{Kea91} regarding the existence of the asymptotic frequency of the symbol $1$ in $\mathbb{K}$, conjectured to be $1/2$. The current best rigorous bounds are $[0.49992, 0.50008]$, established by Rao \cite{Rao12}. 

To gain insight into the mechanisms governing self-reading sequences in a more general framework, let us formalize the action of the run-length encoding operator $\Delta$. For a non-stationary sequence $x=(x_n)_{n \in \mathbb{N}}$ over a finite alphabet $\Sigma \subset \mathbb{N}$, $\Delta$ maps $x$ to the sequence of the lengths of its consecutive mono-symbol blocks. If $x$ is non-stationary, its \textbf{derivative} $\Delta(x)$, also denoted $\Delta x$,  is naturally a sequence over $\mathbb{N}$. Following the concept of \emph{smooth sequences} (or $C^\infty$-sequences) introduced in \cite{BL2003, BDLV06}, we say that $x$ is \emph{differentiable} over $\Sigma$ if $\Delta x \in \Sigma^\mathbb{N}$. Furthermore, if $x$ is infinitely differentiable over $\Sigma$ -- meaning that the $k$-th iterated derivative $\Delta^k x $ belongs to $\Sigma^\mathbb{N}$ for all $k \ge 0$ --  then $x$ is said to be \emph{smooth} (or $C^\infty$) over $\Sigma$.

Dekking \cite{Dek80, Dek81} established fundamental connections between the combinatorial properties of $C^\infty$-words and the Oldenburger-Kolakoski sequence. He demonstrated that the conjecture of \emph{mirror invariance} implies that $\mathbb{K}$ is \emph{recurrent}. Furthermore, he proved that the mirror invariance of $\mathbb{K}$ is equivalent to the statement that every smooth word is \emph{admissible}, thereby identifying the language of $\mathbb{K}$ with the entire class of $C^\infty$-words.
In this context, $C^\infty$-words are defined as finite words whose successive derivatives eventually lead to the empty word, provided that the first and last mono-symbol blocks are ignored during the process if they are incomplete \cite{Dek80}. Recent progress by Cassaigne and Henry \cite{cassaigne2026} provides new bounds for the complexity of these words, and formalizes the class of \emph{f-smooth} words.

A fundamental property of the set of smooth sequences is its uncountability, which is best understood through the \emph{column sequences} map (or vertical sequences) $\Phi$. For any smooth sequences $x$, the column sequences $\Phi(x) \in \Sigma^\mathbb{N}$ is defined by:
\[ \forall k \ge 0, \quad  \Phi(x)_k = \left(\Delta^k (x)\right)_0.\]
In other words, the $k$-th letter of $\Phi(x)$ corresponds to the first letter of the $k$-th derivative of $x$, thus recording the ``vertical'' evolution of the word's prefix under successive derivations (see Figure~\ref{fig:column_word}).

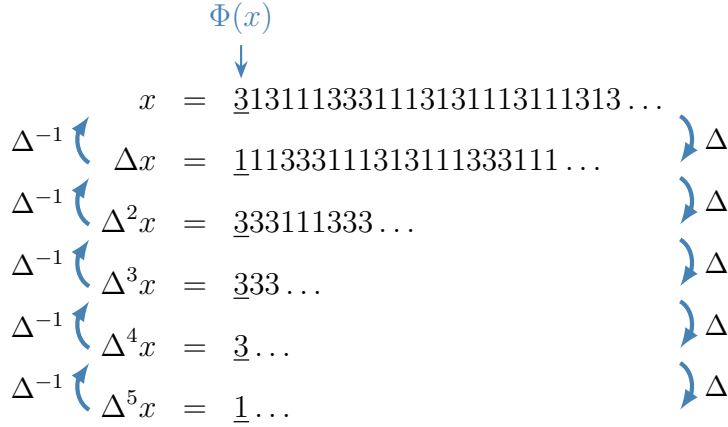
\begin{figure}[!ht]
	{\addtolength{\jot}{0.5em} 
	\begin{eqnarray*}
		\tikzmark{debLigne6} 		x			&	=	&	 \tikzmark{haut}\underline{3}   1   3   111   333   111   3   1   3   111   3  111  3  1  3 \dots \tikzmark{finLigne6} \\
		\tikzmark{debLigne5} 		\Delta x 			&	=	&	 \underline{1}11  333  111   3   1  3  111  333   111  \dots  \tikzmark{finLigne5}   \\ 
		\tikzmark{debLigne4} 		\Delta^2 x 			&	=	&	 \underline{3} 33  111  333  \dots  \tikzmark{finLigne4}   \\
		\tikzmark{debLigne3} 		\Delta^3 x 			&	=	&	 \underline{3} 33 \dots  \tikzmark{finLigne3}  \\
	\tikzmark{debLigne2} 			\Delta^4 x 			&	=	&	 \underline{3}  \dots \tikzmark{finLigne2}  \\ 
	\tikzmark{debLigne1} 	\Delta^5 x 			&	=	&	\tikzmark{bas}\underline{1} \dots \tikzmark{finLigne1}
	\end{eqnarray*}
}

	\begin{tikzpicture}[remember picture, overlay, >=Stealth, SteelBlue]
		\coordinate (XdebLigne1) at ([xshift=-0.15cm] pic cs:debLigne1);
		\coordinate (XdebLigne2) at ([xshift=-0.15cm] pic cs:debLigne2);
		\coordinate (XdebLigne3) at ([xshift=-0.15cm] pic cs:debLigne3);
		\coordinate (XdebLigne4) at ([xshift=-0.15cm] pic cs:debLigne4);
		\coordinate (XdebLigne5) at ([xshift=-0.15cm] pic cs:debLigne5);
		\coordinate (XdebLigne6) at ([xshift=-0.15cm] pic cs:debLigne6);
		
		\coordinate (XfinLigne6) at ([xshift=0.15cm] pic cs:finLigne6);
								
		\coordinate (yD1)   at (pic cs:debLigne1);
		\coordinate (yD2)   at (pic cs:debLigne2);
		\coordinate (yD3)   at (pic cs:debLigne3);
		\coordinate (yD4)   at (pic cs:debLigne4);
		\coordinate (yD5)   at (pic cs:debLigne5);
		\coordinate (yD6)    at (pic cs:debLigne6);
		
		\coordinate (start) at (pic cs:haut);
		\node (phi) at ([yshift=1.2cm, xshift=0.1cm]start) {$\Phi(x)$};
		\draw[->, thick] (phi) -- ([yshift=0.4cm, xshift=0.1cm]start);
		

	\draw[-latex, ultra thick, SteelBlue] ([yshift=1mm]XdebLigne1 |- yD1)  to[bend left=60] node[midway, left, black, font=\small] {$\Delta^{-1}$} ([yshift=-1mm]XdebLigne1 |- yD2);	
	\draw[-latex, ultra thick, SteelBlue] ([yshift=1mm]XdebLigne1 |- yD2)  to[bend left=60] node[midway, left, black, font=\small] {$\Delta^{-1}$} ([yshift=-1mm]XdebLigne1 |- yD3);	
	\draw[-latex, ultra thick, SteelBlue] ([yshift=1mm]XdebLigne1 |- yD3)  to[bend left=60] node[midway, left, black, font=\small] {$\Delta^{-1}$} ([yshift=-1mm]XdebLigne1 |- yD4);	
	\draw[-latex, ultra thick, SteelBlue] ([yshift=1mm]XdebLigne1 |- yD4)  to[bend left=60] node[midway, left, black, font=\small] {$\Delta^{-1}$} ([yshift=-1mm]XdebLigne1 |- yD5);	
	\draw[-latex, ultra thick, SteelBlue] ([yshift=1mm]XdebLigne1 |- yD5)  to[bend left=60] node[midway, left, black, font=\small] {$\Delta^{-1}$} ([yshift=-1mm]XdebLigne1 |- yD6);	
	
	\draw[-latex, ultra thick, SteelBlue] ([yshift=-1mm]XfinLigne6 |- yD6)  to[bend left=60] node[midway, right, black, font=\small] {$\Delta$} ([yshift=1mm]XfinLigne6 |- yD5);	
	\draw[-latex, ultra thick, SteelBlue] ([yshift=-1mm]XfinLigne6 |- yD5)  to[bend left=60] node[midway, right, black, font=\small] {$\Delta$} ([yshift=1mm]XfinLigne6 |- yD4);	
	\draw[-latex, ultra thick, SteelBlue] ([yshift=-1mm]XfinLigne6 |- yD4)  to[bend left=60] node[midway, right, black, font=\small] {$\Delta$} ([yshift=1mm]XfinLigne6 |- yD3);	
	\draw[-latex, ultra thick, SteelBlue] ([yshift=-1mm]XfinLigne6 |- yD3)  to[bend left=60] node[midway, right, black, font=\small] {$\Delta$} ([yshift=1mm]XfinLigne6 |- yD2);	
	\draw[-latex, ultra thick, SteelBlue] ([yshift=-1mm]XfinLigne6 |- yD2)  to[bend left=60] node[midway, right, black, font=\small] {$\Delta$} ([yshift=1mm]XfinLigne6 |- yD1);		
	\end{tikzpicture}
	\caption{The map $\Phi$ provides a one-to-one correspondence between the set of smooth sequences and the set of all infinite sequences over $\Sigma$ \cite{BJP08}. \label{fig:column_word}}
\end{figure}

The combinatorial properties of smooth sequences remain remarkably under-explored. Indeed, only a few combinatorial and dynamical properties have been formally established to date, and these results are strictly limited to a small number of specific alphabets of uniform parity. In these specific cases, where all symbols are either even or odd, the structural regularity and internal symmetries allow for a more systematic analysis. In contrast, mixed-parity alphabets lack such uniform constraints, making them significantly more complex; consequently, most fundamental questions regarding them remain entirely open:

Let $w \in \Sigma^*$ be a finite word of length $|w|$. An occurrence of $w$ in a sequence $x \in \Sigma^* \cup \Sigma^{\mathbb{N}} \cup \Sigma^{\mathbb{Z}}$ is an index $i$ such that $x_{i+k} = w_k$ for all $0 \le k < |w|$. We denote by $|x_0 \cdots x_{n-1}|_w$ the number of occurrences of $w$ in the prefix of length $n$ of $x$.
The \emph{frequency (of occurrence)} of $w$ in $x$ is defined as the limit:
\begin{equation}
 \label{eq:def_freq}
 \mathrm{freq}_x(w) := \lim_{n \to \infty} \frac{|x_0 \cdots x_{n-1}|_w}{n},
\end{equation}
provided that this limit exists.

\begin{enumerate}
	\item \textbf{Uniform Parity (Even or Odd):} When all letters in $\Sigma$ share the same parity, the system becomes more tractable, though results are currently limited to a few specific classes. 
\begin{itemize}
    \item \textit{Even Alphabets:} The structure is remarkably rigid. If $a$ and $b$ are two non-zero even integers, every smooth word over $\{a, b\}$ possesses a symbol frequency of $1/2$. This property is a direct consequence of the reconstruction process via the column map $\Phi$. Starting from a single even letter $\ell$ at some derivation level, its first pre-image $\Delta^{-1} \ell$ is a monochromatic block ($a^\ell$ or $b^\ell$). Because $\ell$ is even, the second pre-image $\Delta^{-2} \ell$ is composed of $\ell/2$ pairs of alternating blocks, taking one of the following forms:
\[ (a^a b^a)^{\ell/2}, \quad (b^a a^a)^{\ell/2}, \quad (a^b b^b)^{\ell/2}, \quad \text{or} \quad (b^b a^b)^{\ell/2}. \]
Each of these patterns contains an equal number of occurrences of $a$ and $b$. Since $a$ and $b$ are themselves even, this perfect local balance is invariant under all subsequent applications of $\Delta^{-1}$. This ensures that the letters remain perfectly equidistributed throughout the ``integration'' process, regardless of the specific vertical sequence $\Phi$ used to reconstruct the word.
	
\item \textit{Odd Alphabets:} These systems exhibit an intermediate behavior that remains structured but is no longer universally balanced. For instance, over the alphabet $\Sigma = \{1, 3\}$, Brlek, Jamet, and Paquin \cite{BJP08} proved that lexicographically extremal smooth sequences possess a specific symbol frequency. For instance, the frequency of $1$s in the minimal smooth word with respect to the lexicographical order $m_{\{1,3\}}$ is $\frac{\sqrt{5}}{\sqrt{5}+1} \approx 0.62$. While these results describe specific extremal cases, the distribution of symbols in the entire set of smooth words over odd alphabets is not yet fully understood. Baake and Sing \cite{Baake_Sing_2004} showed that the Kolakoski-(3,1) sequence relates to model sets with pure point diffraction spectra, providing a structural bridge between these sequences and quasicrystals. In this paper, we aim to conduct a more general study of letter frequencies for all smooth sequences over the alphabet $\Sigma=\{1,3\}$.
\end{itemize}
	
	\item \textbf{Mixed Parity:} This category, which includes the classical alphabet $\{1, 2\}$, represents the primary challenge in the field. Although the column map $\Phi$ establishes that the set of smooth words is uncountable, this case lacks the structural regularity induced by uniform parity. In this mixed setting, virtually all key dynamical properties—such as unique ergodicity, recurrence, or minimality—remain long-standing conjectures that have resisted proof for decades.
\end{enumerate}

 To address fundamental questions for odd alphabets, we extend the study of smooth sequences from the classical one-sided sequences to the bi-infinite ones. This transition is essential for our ergodic approach, as it enables us to work within a compact, shift-invariant space: the subshift $X \subset \{1,3\}^\mathbb{Z}$ of bi-infinite smooth sequences. Our main result is that $X$ can be described as the disjoint union of smaller subshifts
 \[
   X = \bigsqcup_{\tau \in \{0,1\}^\NN} X_\tau,
 \]
 where each $X_\tau$ is uniquely ergodic. 
 This result implies that the asymptotic frequency of any finite pattern is well-defined for both bi-infinite and one-sided smooth sequences over $\{1,3\}$. The proof involves recoding smooth sequences over a new alphabet $\{\A,\B,\C,\D\}$ (Section~\ref{sec:recoding}), giving rise to recoding subshifts $(Y_\tau)_{\tau \in \{0,1\}^\NN}$ with a nice substitutive structure (Section~\ref{sec:substitutions}). Then we establish the unique ergodicity of the subshifts $Y_\tau$ (Section~\ref{sec:UE_Y}), from which we get the unique ergodicity of $X_\tau$ (Section~\ref{sec:UE_X}). We also discuss the minimality of $X_\tau$ in Section~\ref{sec:minimal}.
 
\section{Objects and tools}\label{sec:objects_tools}

    \subsection{Symbolic dynamics and ergodic theory}
    \label{sec:basic}
    We recall below the basic notions in symbolic dynamics and ergodic theory that we will use in the paper. 
    Let $\Sigma$ be a finite alphabet. The space $\Sigma^\ZZ$ of bi-infinite sequences over $\Sigma$ is endowed with the product topology, which makes it a compact metrizable space. For a finite word $w = w_0 w_1\dots w_{\ell - 1} \in \Sigma^*$, we denote by $[w]$ the \emph{cylinder set} 
    \[ [w] := \{x \in \Sigma^\ZZ \,|\, \forall i = 0,\ldots,\ell-1,\ x_i = w_i \}, 
    \]
    which is a clopen subset of $\Sigma^\ZZ$ (in particular, $\ind{[w]}$ is continuous). 
    The shift map $S:\Sigma^\ZZ\to\Sigma^\ZZ$ defined by $(Sx)_n := x_{n+1}$ for all $x\in\Sigma^\ZZ$ and all $n\in\ZZ$ is a bijective continuous transformation of $\Sigma^\ZZ$. 
    A nonempty subset $X\subset \Sigma^\ZZ$ is said to be a \emph{subshift} if it is closed and satisfies $S(X) = X$. In this case, the pair $(X,S)$ is called a \emph{symbolic dynamical system}, which is a particular case of a \emph{topological dynamical system} (defined by the action of a homeomorphism on a compact metrizable space). The subshift $X$ is said to be \emph{minimal} if it contains no proper subshift. 
    
    For a subshift $X$, we denote by $\M_1(X)$ the space of all Borel probability measures on $X$. Endowed with the topology of weak convergence, $\M_1(X)$ is itself a compact metrizable space. 
    A probability measure $\mu\in\M_1(X)$ is said to be $S$-invariant if $\mu(S^{-1}A) = \mu(A)$ for all Borel set $A\subset X$ (equivalently, $\int_X f\,d\mu = \int_X f\circ S\,d\mu$ for all $f\in C(X)$). By Kryloff-Bogolyuboff theorem, the space $\M_1(X,S)$ of $S$-invariant probability measures is never empty, and a classical proof of this result is based on the following fact: for $x \in X$ and $n\ge 1$, we define the \emph{empirical measure} $m_x^n\in\M_1(X)$ as the average of the Dirac masses on the first $n$ points of the positive orbit of $x$:
    \[
    m_x^n := \frac{1}{n} \sum_{j = 1}^n \delta_{S^nx},
    \]
    then any accumulation point of the sequence of empirical measures $\bigl(m_x^n\bigr)_{n\ge 1}$ is necessarily $S$-invariant. 
    If $x\in X$ is such that $\bigl(m_x^n\bigr)_{n\ge 1}$ has a unique accumulation point $\mu$, then we say that \emph{$x$ is generic for $\mu$}. Note also that the convergence $m_x^n \to \mu$ is equivalent to the property
    \[
    \forall w\in\Sigma^*,\ \mathrm{freq}_x(w) = \lim_{n\to\infty} \frac{1}{n} \sum_{j = 1}^n \ind{[w]}(S^nx) = \lim_{n\to\infty} m_x^n([w]) = \mu([w]).
    \]
    
    The subshift $X$ is said to be \emph{uniquely ergodic} if $\M_1(X,S)$ is reduced to a singleton $\{\mu\}$. In this case, every $x \in X$ is generic for $\mu$. The converse is also true: if there exists some $\mu \in \M_1(X,S)$ such that every $x\in X$ is generic for $\mu$, then $\mu$ is the only $S$-invariant probability measure on $X$, and moreover $\mu$ is ergodic (that is: $\mu$ gives measure 0 or 1 to any $S$-invariant subsets of $X$).
    Indeed, we recall that $\M_1(X,S)$ is a compact convex set whose extremal points are the ergodic measures (see \emph{e.g.} \cite[Section~3.5]{EisnerFarkas2025}). But, by the Birkhoff ergodic theorem, for any ergodic $S$-invariant probability measure $\nu$ on $X$, we know that $\nu$-almost all $x\in X$ satisfies $m_x^n \tend{n}{\infty} \nu$.

    \subsection{Bi-infinite smooth sequences over $\{1,3\}$}

The operator $\Delta$ maps any non-stationary bi-infinite sequence $x$ to the sequence of lengths of its consecutive mono-symbol blocks. To properly index this derived sequence, we adopt the following convention: the value $(\Delta x)_0$ at the origin of the derivative is the length of the mono-symbol block in $x$ that contains the term $x_0$. This alignment is illustrated in Figure~\ref{fig:bi-infinite-smooth-sequences}, where the term at the origin is underlined.

	\begin{figure}[!ht]
	\centering
	\begin{alignat*}{2}
		\phantom{\Delta^{16}}x &= \ldots3331113131\underline{1}1333111311\ldots \\
		\Delta\,
		x &= \ldots3331333111\underline{3}3313331113\ldots \\
		\Delta^{2}x &= \ldots3111333133\underline{3}1333111333\ldots \\
		\Delta^{3}x &= \ldots3133313331\underline{3}1333111333\ldots \\
		\Delta^{4}x &= \ldots3331113131\underline{1}1333111311\ldots \\
		\Delta^{5}x &= \ldots3331333111\underline{3}3313331113\ldots \\
		\Delta^{6}x &= \ldots3111333133\underline{3}1333111333\ldots \\
		\Delta^{7}x &= \ldots3133313331\underline{3}1333111333\ldots 
	\end{alignat*}
	\caption{Successive derivatives of a bi-infinite smooth sequence $x$. The underlined terms indicate the position of the origin (index $0$) for each sequence, illustrating the alignment convention.}
	\label{fig:bi-infinite-smooth-sequences}
\end{figure}

Having defined the action of the derivative operator $\Delta$ on bi-infinite sequences, we now formalize the concepts of differentiability and smoothness in this setting:

\begin{definition}
	A bi-infinite sequence $x \in \{1,3\}^\mathbb{Z}$ is said to be \emph{differentiable} if its derivative $\Delta x$ is well-defined (i.e., $x$ is non-stationary) and belongs to $\{1,3\}^\mathbb{Z}$. Furthermore, $x$ is said to be \emph{smooth} if it is infinitely differentiable.
\end{definition}

To ensure that the set of bi-infinite smooth sequences is non-empty, let us notice that the column sequence $\Phi(w)$ provides a natural mechanism to uniquely extend any one-sided smooth sequence $w \in \{1,3\}^\mathbb{N}$ to the left. Because one-sided smooth sequences over a two-letter alphabet consist of alternating mono-symbol blocks, embedding $w$ into a bi-infinite sequence $x$ (such that $x_n = w_n$ for $n \ge 0$) implies that the term just before the origin must differ from the one at the origin ($x_{-1} \neq x_0$). This principle applies to every derivation level. Since $\Phi(w)_k$ is the first letter of the $k$-th derivative $\Delta^k w$, the term immediately preceding it in the bi-infinite extension is necessarily its complement, namely $(\Delta^k x)_{-1} = \overline{\Phi(w)}_k$ (where $\overline{1}=3$ and $\overline{3}=1$). This allows for a unique, step-by-step reconstruction of the sequence at all negative indices. While this formally establishes that the set of bi-infinite smooth sequences is non-empty, this set is actually much richer and contains many other sequences beyond those generated through this specific extension process.

Let us denote by $X \subset \{1,3\}^\mathbb{Z}$ the set of all bi-infinite smooth sequences, and by $S$ the standard shift operator defined by $(Sx)_n = x_{n+1}$ for all $n \in \mathbb{Z}$. The set $X$ is obviously invariant under $S$. Furthermore, $X$ is a closed subset of the product space $\{1,3\}^\mathbb{Z}$, which is compact by Tychonoff's theorem. As a closed and shift-invariant subset of a compact space, $X$ constitutes a compact \emph{subshift}, providing the exact topological framework required for our ergodic study.

\subsubsection{Integration of smooth sequences}

One may ask how to reconstruct a bi-infinite smooth sequence $x$ using only its vertical trace. The standard column sequence $\Phi(x)$ is insufficient for this purpose since it lacks alignment information. Indeed, while a value of $1$ at the origin of $\Delta^k x$ unambiguously implies a block of length $1$ in $\Delta^{k-1} x$, a value of $3$ generates a block of length $3$. In the latter case, the origin of $\Delta^{k-1} x$ could be aligned with the \textbf{Left} ($\mathtt{L}$), \textbf{Middle} ($\mathtt{M}$), or \textbf{Right} ($\mathtt{R}$) element of this block.

To resolve this ambiguity, we introduce the notion of a \emph{generalized column sequence}, where each term $3$ is annotated with a position tag $p \in \{\mathtt{L}, \mathtt{M}, \mathtt{R}\}$. This enriched structure allows us to perform the \emph{integration} of the sequence from top to bottom. Starting from a single value at a high derivative level, we can iteratively expand each level to the next, using the parity of indices to determine the alternating values and the position tags to fix the alignment of the origin. Figure~\ref{fig:integration} illustrates this reconstruction (integration) of a bi-infinite smooth sequence $x$ from a generalized column sequence.

\begin{figure}[!ht]
	\centering
\begin{alignat*}{2}
	\phantom{\Delta^{6}}x &= \dots 133311\underline{1}3331313331113331 \dots \\
	\Delta^{1}x &=   \phantom{0}\phantom{0}\phantom{0}\phantom{0} \dots 13\underline{3}31113331 \dots \phantom{0}\phantom{0}\phantom{0}\phantom{0}\phantom{0}\phantom{0}\phantom{0}\phantom{0} \\
	\Delta^{2}x &= \phantom{0}\phantom{0}\phantom{0}\phantom{0}\phantom{0} \dots 1\underline{3}331 \dots \phantom{0}\phantom{0}\phantom{0}\phantom{0}\phantom{0}\phantom{0}\phantom{0}\phantom{0}\phantom{0}\phantom{0}\phantom{0}\phantom{0}\phantom{0} \\
	\Delta^{3}x &= \phantom{0}\phantom{0}\phantom{0}\phantom{0}\phantom{0} \dots 1\underline{3}1 \dots \phantom{0}\phantom{0}\phantom{0}\phantom{0}\phantom{0}\phantom{0}\phantom{0}\phantom{0}\phantom{0}\phantom{0}\phantom{0}\phantom{0}\phantom{0}\phantom{0}\phantom{0} \\
	\Delta^{4}x &= \phantom{0}\phantom{0}\phantom{0}\phantom{0}\phantom{0} \dots 1\underline{1}1 \dots \phantom{0}\phantom{0}\phantom{0}\phantom{0}\phantom{0}\phantom{0}\phantom{0}\phantom{0}\phantom{0}\phantom{0}\phantom{0}\phantom{0}\phantom{0}\phantom{0}\phantom{0} \\
	\Delta^{5}x &= \phantom{0}\phantom{0}\phantom{0}\phantom{0}\phantom{0}\phantom{0} \dots \underline{3} \dots \phantom{0}\phantom{0}\phantom{0}\phantom{0}\phantom{0}\phantom{0}\phantom{0}\phantom{0}\phantom{0}\phantom{0}\phantom{0}\phantom{0}\phantom{0}\phantom{0}\phantom{0}\phantom{0} \\
	\Delta^{6}x &= \phantom{0}\phantom{0}\phantom{0}\phantom{0}\phantom{0}\phantom{0} \dots \underline{1} \dots \phantom{0}\phantom{0}\phantom{0}\phantom{0}\phantom{0}\phantom{0}\phantom{0}\phantom{0}\phantom{0}\phantom{0}\phantom{0}\phantom{0}\phantom{0}\phantom{0}\phantom{0}\phantom{0}
\end{alignat*}
	\caption{Reconstruction (integration) of a smooth bi-infinite sequence $x$ from the generalized column sequence  
	$\bigl(1, 3_{\mathtt{R}}, 3_{\mathtt{M}}, 3_{\mathtt{L}}, 1, 3_{\mathtt{M}}, 1,\dots\bigr)$.
	The underlined term indicates the origin (index 0) at each level. The alignment is determined by the position tags ($\mathtt{L}/\mathtt{M}/\mathtt{R}$) associated with blocks of length 3.}
	\label{fig:integration}
\end{figure}
     
\subsection{Recoding of a smooth sequence}
\label{sec:recoding}
    \subsubsection{Subshift of recoding sequences} 
Recall that we denote by $X \subset \{1,3\}^\mathbb{Z}$ the subshift of bi-infinite smooth sequences over the alphabet $\Sigma=\{1,3\}$. Since the derivative of a smooth sequence is also a sequence over the same alphabet, any $x \in X$ is necessarily a concatenation of blocks of lengths $1$ and $3$. Consequently, a smooth sequence is composed of only four types of mono-symbol blocks, which we recode using the alphabet $\mathcal{A} = \{\A, \B, \C, \D\}$ defined as follows:
\begin{equation*}
  	\A = 1, \quad \B = 3, \quad \C = 111, \quad \D = 333.
\end{equation*}
    
The \emph{recoding} of a smooth sequence $x \in X$ is the sequence $y = \texttt{rec}(x) \in \mathcal{A}^\mathbb{Z}$ where each term $y_i$ encodes a consecutive mono-symbol block (or run) in $x$. By convention, $y_0$ is the symbol encoding the block that contains $x_0$. We denote by $Y = \rec(X) \subset \mathcal{A}^\mathbb{Z}$ the subshift of all such recodings.

An example of this correspondence is shown below, where the underlined symbol indicates the zero position of the sequence:
\[\arraycolsep=2pt
\begin{array}{lcrccccccccl}
	x & = & \cdots 333 &  1\underline{1}1  &  333  &   1 &  3 &  1  &  333  &  111  &  333  &  1\cdots \\
    y = \text{rec}(x) & = & \cdots \D \; & \; \underline{\C} \; &  \; \D \; &   \A &  \B &   \A &  \; \D \; &  \; \C  \; &  \; \D \;&   \A \cdots
\end{array}
\]
 
Observe that each symbol in the recoding sequence $y = \rec(x)$ corresponds to a single symbol in the derivative $\Delta(x)$: a $1$ in $\Delta(x)$ gives rise either to an $\A$ or to a $\B$ in $y$, whereas a $3$  in $\Delta(x)$ gives rise either to a $\C$ or to a $\D$ in $y$. Now, the structuring of $\Delta(x)$ into mono-symbol blocks yields a central property of $y = \rec(x)$: its decomposition into so-called \emph{elementary blocks}, described in the following array.
\[
\begin{array}{ccc}
\text{mono-symbol block in }\Delta(x) & \text{corresponding factor in }x & \text{elementary block in }y \\
 1  & 1 \text{ or }3 &  \A \text{ or } \B \\
 3  & 111 \text{ or }333 &  \C  \text{ or } \D  \\
 111  & 131 \text{ or }313 &  \A\B\A  \text{ or } \B\A\B  \\
 333  & 111333111 \text{ or }333111333 &  \C\D\C  \text{ or } \D\C\D  
\end{array}
\]
Note that these elementary blocks are either factors of $y$ of length 3, of the form $\A\B\A$, $\B\A\B$, $\C\D\C$, $\D\C\D$, or single letters $\A$, $\B$, $\C$, $\D$ that are not part of larger elementary blocks. Observe also that, since mono-symbol blocks in $\Delta(x)$ have length at most 3, elementary blocks in $y$ cannot overlap, and therefore there is a unique way of decomposing $y$ as a bi-infinite concatenation of elementary blocks.

\begin{remark}
Note that the recoding sequence $\rec(x) \in \A^\ZZ$ is well defined as soon as $x \in \{1,3\}^\ZZ$ is differentiable. So we may use $\rec(x)$ for differentiable sequences $x$ whose smoothness is \textit{a priori} unknown, but then $\rec(x)$ is not necessarily in $Y$.
\end{remark}

\subsubsection{The induced derivative operator}
    
    The derivative operator $\Delta$ acting on $X$ naturally induces an operator on $Y$,  also denoted by $\Delta$ and defined by 
    \[ 
    \Delta(\rec(x)) := \rec(\Delta(x)).
    \]
This definition ensures that the derivation and recoding operations commute.
      
In practice, the derivative $\Delta y$ can be computed directly from $y \in Y$ by applying the following set of local rewriting rules acting on elementary blocks of $y$, without reconstructing its pre-image $x \in X$. 
    \begin{equation*}
    	\Delta : \left| \begin{array}{lcl}
    		\A\B\A, \B\A\B & \longmapsto & \C \\
    		\C\D\C, \D\C\D & \longmapsto & \D \\
    		\A, \B & \longmapsto & \A \\
    		\C, \D & \longmapsto & \B
    	\end{array} \right.
    \end{equation*}    
Again, the zero-coordinate of $\Delta y$ corresponds to the elementary block containing the zero-coordinate of $y$. 
    
Let us illustrate this process using the sequence from the previous example:    
    \[
	\begin{tikzcd}[row sep=2.5cm, column sep={1.2cm}]
		x = \cdots 333  1 \underline{1}1   333     1   3   1    333    111    333  \cdots 
		\arrow[r, "\Delta"] \arrow[d, "\rec"'] 
		& 
		\Delta x = \cdots 3 \underline{3} 3  111  333 \cdots 
		\arrow[d, "\rec"] 
		\\
\phantom{333  1 1   333}y = \cdots  \D \underline{\C}     \D    \A    \B    \A    \D    \C    \D    \A \cdots 
		\arrow[r, "\Delta"'] 
		& 
\Delta y = \cdots \underline{\D}  \C  \D \cdots \phantom{{} 111  333} 
	\end{tikzcd}
    \]
 
\subsubsection{Alternating sequences}

Any smooth sequence $x \in X$ alternates between mono-symbol blocks in $\{1,111\}$ and mono-symbol blocks in $\{3,333\}$. Consequently, any recoding $y \in Y$ is \emph{alternating} in the following sense:

\begin{definition}
\label{def:alternating}
      A sequence $y \in \{\A,\B,\C,\D\}^\ZZ$ is said to be \emph{alternating} if it contains neither two consecutive symbols in $\{\A,\C\}$ nor two consecutive symbols in $\{\B,\D\}$.
\end{definition}

This results in the following lemma. 

\begin{lemma}\label{lemma:P7}
Let $y = \rec(x)\in Y$,  with $x \in X$, be the recoding of a smooth bi-infinite sequence, and let $w$ be a finite factor of $y$. We have
\[ \begin{cases}
 \occu{w}{}{\A} + \occu{w}{}{\C} = \occu{w}{}{\B} + \occu{w}{}{\D} = k & \textup{if } |w| = 2k, \\
	\left\{ \occu{w}{}{\A} + \occu{w}{}{\C},\  \occu{w}{}{\B} + \occu{w}{}{\D}\right\} = \{k,k+1\}  & \textup{if } |w| = 2k+1. 
\end{cases}
\]
Therefore,
\[
\max \left\{ 
\left| \dfrac{|w|_\A + |w|_\C}{|w|} - \dfrac{1}{2}  \right|, \,
\left| \dfrac{|w|_\B + |w|_\D}{|w|} - \dfrac{1}{2}  \right|
\right\}  \leq \dfrac{1}{|w|}. 
\]
\end{lemma}
                
Consequently, while the asymptotic frequencies of letters $1$ and $3$ remain to be determined, this bound proves that the aggregated frequency of letters encoding blocks of $1$s (i.e., $\A$ and $\C$) converges to $1/2$ as $|w| \to \infty$.

\subsubsection{Commutation formula}\label{subsubsection::Commutation}

Observe that the derivative operator $\Delta$ and the shift $S$ do not generally commute. To illustrate this, consider the sequence $y = \cdots \A\B\A  \underline{\C}\D\C  \A\B\A \cdots$.
\begin{itemize}
	\item On the one hand, deriving $y$ gives $\Delta y = \cdots \C  \underline{\D}  \C \cdots$, and subsequently shifting it yields $S(\Delta y) = \cdots \C  \D \underline{\C} \cdots$.
	\item On the other hand, shifting $y$ first gives $S(y) = \cdots \A\B\A  \C  \underline{\D}\C  \A\B\A \cdots$. Applying the derivation to this shifted sequence yields $\Delta(Sy) = \cdots \C  \underline{\D} \C \cdots$, which differs from $S(\Delta y)$.
\end{itemize}

This non-commutativity occurs because a single shift in $Y$ moves the origin to the next letter, which may still belong to the same aggregated pattern (here, the origin moved from $\C$ to $\D$ within the $\C\D\C$ block). Consequently, the origin in the derived sequence does not move at all. However, it is possible to relate the shift of a derived sequence to the derivation of a shifted sequence by carefully tracking these block boundaries.

Just as with the original sequences in $X$, the concept of being ``well-aligned'' with the origin is crucial in $Y$: we say that $y \in Y$ is \emph{well-aligned} if the elementary block containing its zero-coordinate starts at the zero position, and we denote by $B_Y$ the subset of $Y$ constituted of well-aligned sequences. 
More generally, for any $L \geq 1$, we define the set of sequences that remain well-aligned over their first $L$ successive derivatives:
\[
B_Y^L := \{y \in Y \mid y, \Delta y, \ldots, \Delta^{L-1} y  \text{ are well-aligned} \}.
\]

To quantify how many times the sequence aligns with these boundaries under successive shifts, we introduce the counting function. For any $y \in Y$, $L \geq 1$, and $n \geq 1$, let:
\[
\Sigma_y^L(n) := \sum_{k=1}^n \ind{B_Y^L}(S^k y)
\]
where $\ind{A}$ denotes the indicator function of the set $A$. By convention, we extend this to $n = 0$ by setting $\Sigma_y^L(0) = 0$. For brevity, we will often use the shorthand $\Sigma_y(n) := \Sigma_y^1(n)$. This counting function leads directly to the following fundamental property, which we refer to as the \emph{commutation formula}:
\begin{proposition}[Commutation formula]
\label{prop:commutation}
For all $L \geq 1$, $n \geq 0$, and $y \in Y$, the following relation holds:
\begin{equation}
	\label{eq:commutation}
	\Delta^L(S^n y) = S^{\Sigma_y^L(n)}(\Delta^L y).
\end{equation}
\end{proposition}

\begin{proof}
Let $\mathcal{P}_L$ denote the property \eqref{eq:commutation} for a fixed $L \geq 1$. We proceed by induction on $L$.

\paragraph{Base case ($L=1$).}
First, observe that by definition of the derivative and the set $B_Y$, the shift operator satisfies the local relation:
\[
\Delta(S y) = \begin{cases}
	S(\Delta y) & \text{if } Sy \in B_Y, \\
	\Delta y & \text{otherwise.}
\end{cases}
\]
This can be written compactly as $\Delta(S y) = S^{\ind{B_Y}(Sy)}(\Delta y)$. 
By iterating this relation $n$ times, a straightforward induction on $n$ yields:
\[
\Delta(S^n y) = S^{\sum_{k=1}^n \ind{B_Y}(S^k y)}(\Delta y) = S^{\Sigma_y^1(n)}(\Delta y),
\]
which establishes $\mathcal{P}_1$.

\paragraph{Inductive step.}
Assume that $\mathcal{P}_L$ holds for some $L \geq 1$. 
Using the recursive definition of $\Delta^{L+1}$ and applying the induction hypothesis $\mathcal{P}_L$, we have
\begin{align*}
	\Delta^{L+1}(S^n y) &= \Delta\left( \Delta^L(S^n y) \right) \\
	&\stackrel{\mathcal{P}_L}{=} \Delta\left( S^{\Sigma_y^L(n)}(\Delta^L y) \right).
\end{align*}
Let $z := \Delta^L y$ and $N := \Sigma_y^L(n)$, so that the last expression becomes $\Delta(S^N z)$. By applying the base case $\mathcal{P}_1$ to $z$ with shift $N$, we obtain
\begin{equation}
	\label{eq:ind_step}
	\Delta^{L+1}(S^n y) = \Delta(S^N z) = S^{\Sigma_z^1(N)}(\Delta z) = 
	S^{\Sigma_z^1(N)}(\Delta^{L+1} y).
\end{equation}
It remains to show that the exponent $\Sigma_z^1(N) = {\Sigma_{\Delta^L y}^1(\Sigma_y^L(n))}$ in \eqref{eq:ind_step} is equal to $\Sigma_y^{L+1}(n)$. 
By definition, 
\[
\Sigma_{\Delta^L y}^1(\Sigma_y^L(n)) = \sum_{k=1}^{\Sigma_y^L(n)} \ind{B_Y}(S^k(\Delta^L y)).
\]
Now, to each $k \in \{1, \dots, \Sigma_y^L(n)\}$ we associate the first index $j = j(k) \in \{1,\dots,n\}$ such that $\Sigma_y^L(j) = k$. Obviously, we have $S^{j(k)}y \in B_Y^L$, and $k\mapsto j(k)$ defines a bijection between $\{1, \dots, \Sigma_y^L(n)\}$ and  the set 
\[ \bigl\{ j \in \{1,\dots,n\} \mid S^{j}y \in B_Y^L \bigr\}. \]
This allows us to perform a change of variable in the summation:
\begin{align*}
\sum_{k=1}^{\Sigma_y^L(n)} \ind{B_Y}(S^k(\Delta^L y)) &=
\sum_{j=1}^n \ind{B_Y^L}(S^j y)  \ind{B_Y}(S^{\Sigma_y^L(j)}\Delta^L(y)) \\
&= \sum_{j=1}^n \ind{B_Y^L}(S^j y)  \ind{B_Y}(\Delta^L(S^j y)).
\end{align*}
(For the second equality we use again the induction hypothesis.) But, by definition of $B_Y^{L+1}$, the sequence $S^j y$ belongs to $B_Y^{L+1}$ if and only if $S^j y \in B_Y^L$ and $\Delta^L(S^j y) \in B_Y$. Thus, the sum simplifies to 
\[
\sum_{k=1}^{\Sigma_y^L(n)} \ind{B_Y}(S^k(\Delta^L y)) = \sum_{j=1}^n \ind{B_Y^{L+1}}(S^j y) = \Sigma_y^{L+1}(n).
\]
Substituting this back into \eqref{eq:ind_step} yields $\Delta^{L+1}(S^n y) = S^{\Sigma_y^{L+1}(n)}(\Delta^{L+1} y)$, which proves $\mathcal{P}_{L+1}$.
\end{proof}

        \subsubsection{A lower bound for $\Sigma_y^L$}
        
        \begin{lemma}
            \label{lemma:Sigma_min}
            For all $n \in \mathbb{N}$, all $L \geq 1$ and all $y \in Y$, $\Sigma_y^L(n+3^L) \geq \Sigma_y^L(n) + 1$. Consequently, we always have
            \[ \Sigma_y^L(n) \geq \left\lfloor\frac{n}{3^L}\right\rfloor.
            \]
        \end{lemma}
        
        \begin{proof}
            For $L = 1$, the results comes from the fact that elementary blocks in $y \in Y$ have length at most 3. Therefore, for any $n \in \ZZ$, at least one of $S^{n+1}y$, $S^{n+2}y$, $S^{n+3}y$ is well-aligned, hence \[\Sigma_y(n+3) = \Sigma_y(n) + \ind{B_Y}(S^{n+1}y) + \ind{B_Y}(S^{n+2}y) + \ind{B_Y}(S^{n+3}y) \geq \Sigma_y(n) +1.\]
            
            Then the general result follows by induction on $L$, using the commutation formula from Proposition~\ref{prop:commutation} (details are left to the reader).
            
        \end{proof}
        
    \subsection{Type}
    
The arithmetic nature of the alphabet $\Sigma=\{1,3\}$ imposes a strict parity constraint on the positions of the blocks. Since all block lengths are odd, the starting indices of consecutive blocks must alternate in parity.

\begin{proposition}
	\label{prop:index}
	Let $x \in X$ be a smooth sequence over $\{1,3\}$. The starting indices of consecutive mono-symbol blocks in $x$ alternate between even and odd integers. 
	
	In particular, the starting indices of blocks of $1$s all have the same parity, while the starting indices of blocks of $3$s all have the opposite parity.
\end{proposition}

\begin{proof}
	By definition of a smooth sequence, the derivative $\Delta(x)$ is also a sequence over $\{1, 3\}$. Consequently, the length of every mono-symbol block in $x$ is an odd integer.
	Let $i_k$ denote the starting index of the $k$-th block of $x$. The starting index of the next block is given by $i_{k+1} = i_k + \ell_k$, where $\ell_k \in \{1, 3\}$ is the length of the current block. Since $\ell_k$ is odd, it follows that $i_{k+1}$ and $i_k$ have distinct parities ($i_{k+1} \not\equiv i_k \pmod 2$).
\end{proof}

\begin{corollary}
	\label{cor:types}
	A sequence $y \in Y$ cannot simultaneously contain an elementary block in $\{\A\B\A, \D\C\D\}$ and an elementary block in $\{\B\A\B, \C\D\C\}$.
\end{corollary}
    
\begin{proof}
	Let $y = (y_j)_{j \in \ZZ} \in Y$ be the recoding of a smooth sequence $x \in X$. By construction, the sequence $y$ alternates between symbols encoding runs of $1$s (the set $\{\A, \C\}$) and symbols encoding runs of $3$s (the set $\{\B, \D\}$). 
	
	According to Proposition \ref{prop:index}, the starting indices of the mono-symbol blocks in $x$ must alternate in parity. Since each index $j$ of the sequence $y$ refers to the $j$-th block of $x$, it follows that there exists a fixed parity $p \in \{0, 1\}$ (depending on $y$) such that:
	\begin{equation}
	\label{eq:parity}
	\begin{aligned}
		y_j \in \{\A, \C\} &\iff j \equiv p \pmod 2, \\
		y_j \in \{\B, \D\} &\iff j \not\equiv p \pmod 2.
	\end{aligned}
	\end{equation}
	
	Assume, for the sake of contradiction, that $y$ contains both patterns $\A\B\A$ and $\B\A\B$:
	\begin{itemize}
		\item Let some occurrence of $\A\B\A$ begin at index $i_1$. Since $y_{i_1} = \A \in \{\A, \C\}$, we must have $i_1 \equiv p \pmod 2$.
		\item Let some occurrence of $\B\A\B$ begin at index $i_2$. Since $y_{i_2} = \B \in \{\B, \D\}$, we must have $i_2 \not\equiv p \pmod 2$.
	\end{itemize}
	Thus, $i_1$ and $i_2$ necessarily have opposite parities.
	Now recall the rewriting rules for the derivative operator $\Delta$ on $Y$. Both three-letter elementary blocks are mapped to a single symbol representing a block of length 3:
	\[ \Delta(\A\B\A) = \C \quad \text{and} \quad \Delta(\B\A\B) = \C. \]
	Therefore the derived sequence $\Delta y$ contains the symbol $\C$ at two positions $k_1$ and $k_2$, corresponding to $\A\B\A$ in $y$ at position $i_1$ and $\B\A\B$ at position $i_2$. Since elementary blocks always have odd lengths, we must have $i_1 - i_2 \equiv k_1 - k_2 \pmod 2$, and it follows that $k_1$ and $k_2$ also have opposite parities.  
	But this contradicts~\eqref{eq:parity} for $\Delta y$, which proves that the elementary blocks $\A\B\A$ and $\B\A\B$ cannot coexist in $y$. 
	
	Similar arguments prove the impossibility of seeing simultaneously in $y$ the other pairs of elementary blocks indicated in the statement of the corollary.
\end{proof}   
    
This internal consistency allows us to classify smooth sequences into two distinct categories based on their local structure.

\begin{definition}
  	Let $x \in X$ be a smooth sequence and $y = \rec(x)$ its recoding. We say that $x$ and $y$ are of \emph{type~0} if  $y$ contains an elementary block in $\{\A\B\A, \D\C\D\}$. Conversely, we say that $x$ and $y$ are of \emph{type~1} if $y$ contains an elementary block in $\{\B\A\B, \C\D\C\}$. 
  	
  	We denote the type of $x$ and $y$ by $\type(x) = \type(y) \in \{0, 1\}$. The sequence of types of all iterated derivatives is denoted by:
  	\[ \Types(x) = \Types(y) := \bigl(\type\left(\Delta^n x\right)\bigr)_{n \in \mathbb{N}} \in \{0, 1\}^\mathbb{N}. \]
\end{definition}

Note that this classification is exhaustive for smooth sequences over $\{1, 3\}$. Indeed, any recoding $y \in Y$ must contain at least one of the four patterns $\{\A\B\A, \B\A\B, \D\C\D, \C\D\C\}$. If it did not, the derived sequence $\Delta y$ would consist exclusively of the letters $\A$ and $\B$. In terms of the original sequence, this would mean that $\Delta x$ is composed exclusively of blocks of length $1$. Consequently, the second derivative $\Delta^2 x$ would be the constant sequence $1^\infty$, which is not smooth over the alphabet $\{1, 3\}$. Thus, the type is well-defined for every $x \in X$, and Corollary~\ref{cor:types} ensures that it is unique.

\begin{remark}
\label{rem:type_and_parity}
An equivalent way to define the type of $x \in X$, which refers to Proposition~\ref{prop:index}, is the following: $\type(x) = 0$ if the starting indices of blocks of 1s have the same parity in $x$ and $\Delta x$, and $\type(x) = 1$ otherwise. When $\type(x) = 0$, any block $111$ in $\Delta x$ expands to 131 in $x$, and any block 333  in $\Delta x$ expands to $333\,111\,333$ in $x$. When $\type(x) = 1$, any block $111$ in $\Delta x$ expands to 313 in $x$, and any block 333  in $\Delta x$ expands to $111\,333\,111$ in $x$.
\end{remark}

\begin{definition}
      For a sequence of types $\tau \in \{0, 1\}^\mathbb{N}$, we define $X_\tau$ as the set of smooth bi-infinite sequences whose type sequence is exactly $\tau$:
  	\[ X_\tau := \{ x \in X \mid \Types(x) = \tau \}. \]
  	The set of corresponding recodings is denoted by 
  	\[ Y_\tau := \rec(X_\tau) = \{ y \in Y \mid \Types(y) = \tau \}. \]
\end{definition}

\subsection{Substitutive structure of $Y_\tau$}
\label{sec:substitutions}

The classification of smooth sequences into two distinct types provides a natural framework for inverting the derivative operator $\Delta$. Specifically, each type $t \in \{0, 1\}$ allows for the construction of a substitution $\varphi_t$ that reconstructs a recoded sequence from its derivative, up to a bounded shift. 

We define the substitutions $\varphi_0$ and $\varphi_1$ acting on the alphabet $\mathcal{A} = \{\A, \B, \C, \D\}$ as follows:
\[
\varphi_0:\begin{cases}\A \mapsto \A \\ \B \mapsto \D \\ \C \mapsto \A\B\A \\ \D \mapsto \D\C\D \end{cases}     
\qquad      
\varphi_1:\begin{cases}\A \mapsto \B \\ \B \mapsto \C \\ \C \mapsto \B\A\B \\ \D \mapsto \C\D\C \end{cases}.     
\]

Let us consider the action of these substitutions on finite factors whose domain contains the origin. For this purpose,  let $y_{[\![m,n]\!]} = y_m \dots y_{-1} y_0 y_1 \dots y_n$ be a finite word over $\{\A,\B,\C,\D\}$ defined on an integer interval $[\![m,n]\!]$ with $m \le 0 \le n$. Applying $\varphi_t$ on every letter yields the concatenated sequence of blocks $\varphi_t(y_m) \dots \varphi_t(y_0) \dots \varphi_t(y_n)$. To properly index the resulting sequence, we adopt the convention that the block $\varphi_t(y_0)$ begins exactly at index $0$. The positions of all other blocks are then uniquely determined by expanding outwards to the left and right. This ensures that $\varphi_0$ and $\varphi_1$ act as morphisms under concatenation. 

Taking the limits $m \to -\infty$ and $n \to +\infty$, this construction naturally extends to any bi-infinite sequence $y \in \{\A,\B,\C,\D\}^\ZZ$:
\[
\varphi_t (y) = \lim_{n \to \infty} \varphi_t \left( y_{[\![-n,n]\!]}  \right)  
\]

By design, the image of any bi-infinite sequence $y \in Y$ under these substitutions is always well-aligned at the origin. This alignment convention is illustrated in the following commutative diagram:

\[
\begin{tikzcd}[column sep=-3.5em, row sep=2em]
	& y = \cdots \D \underline{\C} \D \A \cdots & \\
	\varphi_0(y) = \cdots \D\C\D \, \underline{\A}\B\A \, \D\C\D \, \A \cdots && \varphi_1(y) = \cdots \C\D\C \, \underline{\B}\A\B \, \C\D\C \, \B \cdots \\
	& y = \cdots \D \underline{\C} \D \A \cdots
	\arrow["\varphi_0"', from=1-2, to=2-1, bend right=15]
	\arrow["\varphi_1", from=1-2, to=2-3, bend left=15]
	\arrow["\Delta"', from=2-1, to=3-2, bend right=15]
	\arrow["\Delta", from=2-3, to=3-2, bend left=15]
\end{tikzcd}
\]

As illustrated above, these two substitutions act as right inverses of the derivative operator $\Delta$: for any $y \in Y$ and any $t \in \{0,1\}$, we have $\Delta(\varphi_t(y)) = y$. Conversely, applying the substitution  $\varphi_{\type(y)}$ to the derivative $\Delta y$ recovers the original sequence $y$ only up to a bounded shift.

To illustrate the necessity of this positional shift, consider a sequence $y \in \{\A,\B,\C,\D\}^\ZZ$ of type 0 where the origin (underlined) falls in the middle of an $\A\B\A$ block, for instance $y =  \cdots \D\C\D \, \A \underline{\B} \A \, \D\C\D \cdots $. Then, 
\[
\varphi_0(\Delta y) = \varphi_0( \cdots \D \underline{\C} \D \cdots ) = \cdots \D\C\D \, \underline{\A} \B \A \, \D\C\D \cdots = S^{-1}(y)
\]
While the sequence of symbols in $\varphi_0(\Delta y)$ perfectly matches $y$, the origin is misaligned because our convention forces the substituted block $\A\B\A$ to start exactly at index $0$. To recover the original sequence $y$, we must shift the result one position to the left.

In general, as substituted blocks have length 1 or 3, the positional shift to recover the original sequence from its derivative can be of 0, 1 or 2 places, and we get the following: 
\begin{equation}
 \label{eq:recover_from_first_derivative}
 \forall y \in Y,\ \exists i  \in \{0, 1,2\}:\ y = S^i\bigl(\varphi_{\type(y)}(\Delta y)\bigr). 
\end{equation}

Now, we want to generalize the above formula to see how $y \in Y$ can be recovered from $\Delta^L y$ for $L \ge 1$. To establish an induction from~\eqref{eq:recover_from_first_derivative}, we need to understand how, for $t \in \{0,1\}$, the substitution $\varphi_t$ interacts with the shift map. It is straightforward to check that for $y \in \{\A,\B,\C,\D\}^\ZZ$, $\varphi_t (Sy) = S^j \varphi_t(y)$, where $j$ is the length of $\varphi_t(y_0)$, which belongs to $\{1, 3\}$. An easy induction then gives:
\begin{equation}
 \label{eq:phi_S}
 \forall y \in  \{\A,\B,\C,\D\}^\ZZ,\ \forall i \ge 1,\ \exists j \in \{i,\ldots,3i\}:\ \varphi_t(S^i y) = S^j \varphi_t(y).
\end{equation}
This allows us to get the following result.

\begin{proposition}
\label{Prop prim}
For all $\tau \in \{0,1\}^\NN$, for all $y \in Y_\tau$, for all $L \ge 1$, there exists $i \in \{0,\dots,3^L-1\}$ such that 
\begin{equation}
    \label{eq:Lsubstitutions}
    y = S^i \varphi_{\tau_0} \circ \varphi_{\tau_1} \circ \cdots \circ \varphi_{\tau_{L-1}} \bigl( \Delta^L y \bigr).
\end{equation}
\end{proposition}

\begin{proof}
 The proof goes by induction on $L$. For $L = 1$, this is just~\eqref{eq:recover_from_first_derivative}. Now, assuming that the result holds up to $L - 1$ for some $L \ge 2$, we apply this induction hypothesis to $\Delta y \in Y_{S\tau}$, getting some $i_1 \in \{0,\dots,3^{L-1} - 1\}$ such that 
\[
  \Delta y = S^{i_1} \varphi_{\tau_1} \circ \varphi_{\tau_2} \circ \cdots \circ \varphi_{\tau_{L-1}} \bigl( \Delta^L y \bigr).
\]
Then we also apply~\eqref{eq:recover_from_first_derivative} to $y$, which gives $i_0 \in \{0,1,2\}$ such that 
\[
  y = S^{i_0} \varphi_{\tau_0} (\Delta y).
\]
Now using~\eqref{eq:phi_S}, we get $j_0 \in \{i_1,\ldots,3i_1\}$ such that 
\[
  \varphi_{\tau_0} (\Delta y) = S^{j_1} \varphi_{\tau_0} \circ \varphi_{\tau_1} \circ \cdots \circ \varphi_{\tau_{L-1}} \bigl( \Delta^L y \bigr),
\]
which yields~\eqref{eq:Lsubstitutions} with $i := i_0 + j_1 \in \{0,\ldots,3^L-1\}$.
\end{proof}

At this point, it is tempting to believe that $Y_\tau$ coincides with the set of all bi-infinite sequences $y \in \mathcal{A}^\ZZ$ satisfying: there exists an infinite family $(y^L)_{L\in\NN}$ of bi-infinite sequences in $\mathcal{A}^\ZZ$ and an infinite family $(i_L)_{L \in \NN}$ of integers such that 
\begin{itemize}
 \item $y = y^0$,
 \item for all $L \ge 0$, $y^L = S^{i_L} \varphi_{\tau_L} \bigl(y^{L+1}\bigr)$,
\end{itemize}
which would amounts to describing $Y_\tau$ as the S-adic subshift associated to the directive sequence of substitutions $\bigl(\varphi_{\tau_L}\bigr)_{L\ge 0}$. (See~\cite{BertheDelecroix2014} for details about S-adic subshifts.) Yet this is not exactly true, at least when $\tau$ ends up with infinitely many 0s, since the sequence 
\[ \dots \A \A \A \underline{\A} \A \A \A  \dots \]
is clearly invariant by $\varphi_0$, but it does not belong to $Y$. To provide a uniform characterization of sequences in $Y_\tau$, we will resort also to the notion of alternating sequences introduced in Definition~\ref{def:alternating}.

\begin{proposition}
 \label{prop:elements_of_Ytau}
 Let $\tau \in \{0,1\}^\NN$ and let $(y^L)_{L \in \NN}$ be an infinite family of alternating sequences in $\mathcal{A}^\ZZ$ satisfying:
 \begin{equation}
 \label{eq:yL}
  \forall L \in \NN,\ \exists i_L \in \ZZ:\ y_L = S^{i_L} \varphi_{\tau_L}  \bigl(y^{L+1}\bigr).
 \end{equation}
 Then $y^0 \in Y_\tau$.
\end{proposition}

\begin{proof}
 As every $y^L$ is alternating, it is straightforward to check that for each $L$ there exists a unique well-aligned and differentiable bi-infinite sequence $x_L \in \{1, 3\}^\ZZ$ such that $y_L = \rec (x_L)$. 
 Moreover, let us assume for example that for some $L$ we have $\tau_L = 0$. Then by construction of $\varphi_{0}$, $y_L = S^{i_L} \varphi_0 \bigl(y^{L+1}\bigr)$ is a concatenation of elementary blocks $\A$, $\D$, $\A\B\A$, $\D\C\D$. To these elementary blocks respectively correspond in $x_L$ the factors $1$, $333$, $131$, $333111333$. By derivation, these factors yield respectively in $\Delta x_L$ the single-symbol blocks $1$, $3$, $111$, $333$, whose recoding give the preimages $\A$, $\B$, $\C$, $\D$ by $\varphi_0$ of the elementary blocks in $y^L$. Therefore the symbols in $\rec(\Delta x_L)$ appear in the same order as those in $y^{L+1}$, up to some possible finite shift.  We finally get that $\Delta x_L = S^{j_L} x_{L+1}$ for some $j_L \in \ZZ$. A similar argument gives the same conclusion when $\tau_L = 1$, and this is enough to conclude that $x^0$ is smooth, so $y^0 \in Y$. Then using~\eqref{eq:yL}, we get that for all $L$, $\type(y_L) = \tau_L$. But by Proposition~\ref{Prop prim} we also have $\Delta^L y = S^{k_L} y^L$ for some $k_L \in \ZZ$, therefore $\type(\Delta^L y) = \type(y^L) = \tau_L$, and finally $\Types(y^0) = \tau$.  
\end{proof}

With the help of the above proposition, we are now able to prove the following important result.

\begin{proposition}
    For any $\tau \in \{0,1\}^\ZZ$, $Y_\tau \neq \emptyset$.
\end{proposition}\label{prop:nonempty}

\begin{proof}
For a fixed $\tau \in \{0,1\}^\ZZ$, we will explain how to construct a specific element of $Y_\tau$, which will justify that $Y_\tau$ is never empty. We construct by induction on $n \in \NN$ an array $(w_n^L)_{0 \le L \le n}$ of finite alternating words over the alphabet $\A$, indexed by subintervals of $\ZZ$ containing 0, with the following properties:
\begin{itemize}
 \item $\forall 0 \le L \le n$, the symbol of index 0 in $w_n^L$ is always a $\D$,
 \item $\forall n$, $w_n^n$ is reduced to the single symbol $\underline{\D}$ (as usual, the underlining is just here to indicate the zero-position),
 \item $\forall 0 \le L < n$, $w_n^L = S^{i_L} \varphi_{\tau_L} \left(w_n^{L+1}\right)$ for some $i_L \in \{0,1,2\}$\quad(*),
 \item $\forall 0 \le L \le n$, $w_{n+1}^L$ extends $w_n^L$,
 \item $\forall L \ge 0$ $w_n^L$ converge as $n \to \infty$ to an alternating bi-infinite sequence $y^L \in \A^\ZZ$.
\end{itemize}
\begin{figure}[htp]
 \centering
 \includegraphics{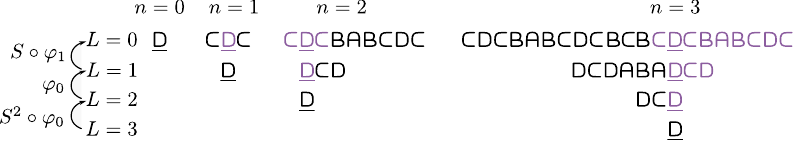}
 \caption{First 4 steps in the construction of the array, for $\tau$ starting with $(1,0,0,\dots)$}
 \label{fig:extensions}
\end{figure}

We start the construction by simply setting $w_0^0 := \underline{\D}$. Now, assume that the array is already constructed up to some $n \ge 0$, and that $i_0,\dots,i_{n-1}$ are already defined. Then we set $w_{n+1}^{n+1} := \underline{D}$. Next we construct $w_{n+1}^n$ extending $w_n^n = \underline{\D}$ with the following rule:
\begin{itemize}
 \item if $\tau_n = 1$, we set $w_{n+1}^n := \C \underline{\D} \C$ and $i_n := 1$,
 \item if $\tau_n = 0$, we have some choice: we can either set $w_{n+1}^n := \underline{\D} \C \D$ and $i_n := 0$ (option~1), or set $w_{n+1}^n := \D \C \underline{\D}$ and $i_n := 2$ (option~2).
\end{itemize}
Let us decide to adopt option~1 whenever $n$ is odd and option~2 whenever $n$ is even. In this way we make sure that the extension of $w_{n+1}^n$ infinitely often adds symbols to the left and to the right. Once we have $w_{n+1}^n$, we get the other words $w_{n+1}^L$ step by step simply applying the requirement~(*) (see Figure~\ref{fig:extensions}).

At the end, the alternating bi-infinite sequences $y^L$ that we get in the limit satisfy 
\[
 \forall L\ge 0,\ y^L = S^{i_L} \varphi_{\tau_L} y^{L+1}.
\]
By Proposition~\ref{Prop prim}, $y^0 \in Y_\tau$.
\end{proof}

\bigskip

The fact that $X_\tau$ is not empty either follows directly. As both $X_\tau$ and $Y_\tau$ are clearly closed and shift-invariant, they are subshifts over respectively the alphabets $\{1, 3\}$ and $\{\A, \B, \C, \D\}$. The next section is devoted to the proof of their unique ergodicity.

\section{Unique ergodicity}\label{sec:unique_ergo}

In all this section, we fix a sequence of types $\tau \in \{0,1\}^\NN$. We want first to prove the unique ergodicity of the subshift $Y_\tau$ by exploiting the substitutive structure highlighted in Section~\ref{sec:substitutions}. Then the unique ergodicity of $X_\tau$ will be deduced from that of $Y_\tau$. 

    \subsection{Unique ergodicity of $Y_\tau$}
    \label{sec:UE_Y}
    Establishing the unique ergodicity of $Y_\tau$ amounts to proving the existence of a probability measure $\nu_\tau$, necessarily $S$-invariant and supported on $Y_\tau$, such that every $y\in Y_\tau$ is generic for $\nu_\tau$, that is:
    \[
     \frac{1}{n} \sum_{j = 1}^{n} \delta_{S^j y} \tend{n}{\infty} \nu_\tau.
    \]
    By compacity of the space $\M_1(Y_\tau)$ 
    of shift-invariant probability measures on $Y_\tau$, equipped with the topology of weak convergence (as $Y_\tau$ itself is compact), it is enough to prove the following proposition.
    
    \begin{proposition}
    \label{prop:accumulation}
     Let $y\in Y_\tau$, and let $\nu$ be an accumulation point of the sequence of empirical measures on the orbit of $y$, that is, for some strictly increasing sequence $(n_k)$
     of integers, we have
     \begin{equation}
     \label{eq:weak_convergence}
     \frac{1}{n_k} \sum_{j = 1}^{n_k} \delta_{S^j y} \tend{k}{\infty} \nu.
    \end{equation} 
    Then for any finite word $w \in \mathcal{A}^*$, 
    \[ \nu([w]) = \lim_{k\to\infty} \frac{1}{n_k}\occu{y}{[1,n_k]}{w}\] 
    depends neither on $(n_k)$ nor on $y\in Y_\tau$.
    \end{proposition}
    
    Indeed, since the cylinder sets of the form $[w]$, $w \in \mathcal{A}^*$, constitute a $\Pi$-system generating the Borel $\sigma$-algebra of $Y_\tau$, it will follow from the above proposition that every $y \in Y_\tau$ is generic for some probability measure $\nu$, and that this $\nu$ does not depend on $y \in Y_\tau$.
    
    \bigskip
    
    For the proof of Proposition~\ref{prop:accumulation}, we now fix $y\in Y_\tau$, $(n_k)$ and $\nu$ satisfying~\eqref{eq:weak_convergence}. 
    Then, as stated in the following lemma, for $L\ge 1$ we can also identify accumulation points of empirical measures on the orbit of $\Delta^L y$ in $Y_{S^L\tau}$\footnote{We also use the same symbol $S$ to denote the unilateral shift map acting on $\{0,1\}^\NN$.}, along the sequence $\left(\Sigma_y^L(n_k)\right)$.
    
    \begin{lemma}
    \label{lemma:DeltaLygeneric}
     For each $L\ge1$, we also have the following weak convergence in $\M_1(Y_{S^L\tau})$:
     \begin{equation}
     \label{eq:weak_convergence2}
     \frac{1}{\Sigma_y^L(n_k)} \sum_{j = 1}^{\Sigma_y^L(n_k)} \delta_{S^j \Delta^L y} \tend{k}{\infty} (\Delta^L)_*\bigl(\nu(\,\cdot\,|\,B^L_{Y_\tau})\bigr),
    \end{equation} 
    where $\nu(\,\cdot\,|\,B^L_{Y_\tau})$ is the probability measure $\nu$ conditioned on $B^L_{Y_\tau} := B^L_Y \cap Y_\tau$, and the limit on the right-hand side is the pushforward of this conditioned measure by $\Delta^L$.
    \end{lemma}
    
    \begin{proof}
     Let $f:\ Y_{S^L\tau}\to\RR$ be a continuous function. Then we define $f^*:Y_\tau\to\RR$ by 
     \[
      f^*(z) := \begin{cases} f(\Delta^L z) & \text{if } z \in B^L_{Y_\tau}, \\
                 0 & \text{otherwise.}
                \end{cases}
     \]
     Observe that, since $\ind{B^L_{Y_\tau}}$ is continuous on $Y_\tau$, $f^*$ is continuous too. Therefore, applying~\eqref{eq:weak_convergence} we get     
     \begin{equation}
     \label{eq:moyenne_fstar}
     \frac{1}{n_k} \sum_{j = 1}^{n_k} f^*(S^j y) \tend{k}{\infty} \int_{Y_\tau} f^*\,d\nu = \int_{B^L_{Y_\tau}} f\circ \Delta^L\, d\nu.
    \end{equation} 
    In the particular case where $f$ is constantly equal to 1, we get
    \begin{equation}
     \label{eq:fconstant1}
     \frac{1}{n_k} \Sigma_y^L(n_k) = \frac{1}{n_k} \sum_{j=1}^{n_k} \ind{B^L_{Y_\tau}}(S^j y) \tend{k}{\infty} \nu\bigl(B^L_{Y_\tau}\bigr).
    \end{equation}
    Note that, by Lemma \ref{lemma:Sigma_min}, $\Sigma_y^L(n_k) \ge n_k / 3^L$ and therefore $\nu(B^L_{Y_\tau}) \ge 3^{-L} > 0$.
    
    Taking into account the commutation formula~\eqref{eq:commutation}, we can rewrite the sum on the left-hand side of~\eqref{eq:moyenne_fstar} as follows:
    \[
      \sum_{j = 1}^{n_k} f^*(S^j y) = \sum_{j = 1}^{n_k} \ind{B^L_{Y_\tau}}(S^j y) f\left(S^{\Sigma^L_y(j)}(\Delta^L y)\right).
    \]
    But since $\Sigma^L_y(j)$ increases by 1 exactly when $S^j y \in B^L_{Y_\tau}$, this sum becomes
    \[
     \sum_{j = 1}^{n_k} f^*(S^j y) = \sum_{h = 1}^{\Sigma^L_y(n_k)} f\bigl(S^h (\Delta^L y)\bigr).
    \]
    Finally, from~\eqref{eq:moyenne_fstar} and~\eqref{eq:fconstant1}, we get
    \begin{multline*}
     \frac{1}{\Sigma^L_y(n_k)} \sum_{h = 1}^{\Sigma^L_y(n_k)} f\bigl(S^h (\Delta^L y)\bigr) = \frac{n_k}{\Sigma^L_y(n_k)} \frac{1}{n_k} \sum_{j = 1}^{n_k} f^*(S^j y) \\  \tend{k}{\infty} \frac{1}{\nu(B^L_{Y_\tau})}\int_{B^L_{Y_\tau}} f\circ \Delta^L\, d\nu,
    \end{multline*}
    which proves the lemma.    
    \end{proof}

        \subsubsection{Induction for the frequencies of letters}
        
        It follows in particular from Lemma~\ref{lemma:DeltaLygeneric} that for any $L \ge 0$, every elementary block in $\mathcal{A}$ has a frequency in $\Delta^L y$ along the sequence $\bigl(\Sigma_y^L(n_k)\bigr)_{k\ge1}$. Let us denote by $a^L$ (respectively $b^L$, $c^L$ and $d^L$) the frequency of the letter $\A$ (respectively $\B$, $\C$ and $\D$) observed on $\Delta^L y$ along $\bigl(\Sigma_y^L(n_k)\bigr)_{k\ge1}$. 
         
        To simplify the notations, let $w_k^0 := y_{[1, n_k]}$, and for $L \ge 1$, $w_k^L := (\Delta^L y)_{[1,\Sigma^L_y(n_k)]}$, so that 
        \[ a^L = \lim_{k\to\infty} \frac{\left| w_k^L \right|_{\A}}{\left| w_k^L \right|,} \]
        and similarly for $b^L$, $c^L$ and $d^L$. Then, it follows from Lemma~\ref{lemma:P7} that 
        \begin{equation}
         \label{eq:undemi}
            a^L + c^L = b^L + d^L = 1/2.
        \end{equation}
        
        From Section~\ref{sec:substitutions}, we know that there exists $i \in \{0, 1, 2\}$ (depending on $L \ge 0$) such that 
        \[
         \Delta^L y = S^i \varphi_{\tau_L} \bigl( \Delta^{L+1}y \bigr).
        \]
        In particular, $w_k^L$ differs from $\varphi_{\tau_L}(w_k^{L+1})$ by at most 2 letters at its beginning and at its end. 
        
        Let us consider some fixed $L \ge 0$ such that $\tau_L = 0$. Then, from the form of the substitution $\varphi_0$, we have for all $k$
        \begin{equation}
         \label{eq:system0}
         \begin{cases}
          \left| w_k^L \right|_{\A} = \left| w_k^{L+1} \right|_{\A} + 2 \left| w_k^{L+1} \right|_{\C} + O(1) \\
          \left| w_k^L \right|_{\B} = \left| w_k^{L+1} \right|_{\C} + O(1) \\
          \left| w_k^L \right|_{\C} = \left| w_k^{L+1} \right|_{\D} + O(1) \\
          \left| w_k^L \right|_{\D} = \left| w_k^{L+1} \right|_{\B} + 2 \left| w_k^{L+1} \right|_{\D} + O(1)
         \end{cases}
        \end{equation}
        Summing the four above equations, we get
        \[
           \left| w_k^L \right| = 3 \left| w_k^{L+1} \right| - 2 \Bigl(\left| w_k^{L+1} \right|_{\A} + \left| w_k^{L+1} \right|_{\B}\Bigr) + O(1).
        \]
        After dividing by $\left| w_k^{L+1} \right|$, we obtain
        \begin{equation}
         \label{eq:rapport}
         \frac{\left| w_k^L \right|}{\left| w_k^{L+1} \right|} = 3 - 2(a^{L+1} + b^{L+1}) + o(1).
        \end{equation}
        On the other hand, dividing the first line of~\eqref{eq:system0} by $\left| w_k^L \right|$ yields
        \[
          a^L = \frac{\left| w_k^{L+1} \right|}{\left| w_k^L \right|} \Bigl( a^{L+1} + 2c^{L+1} \Bigr) + o(1).
        \]
        Taking~\eqref{eq:undemi} and~\eqref{eq:rapport} into account, and taking the limit as $k \to \infty$, we obtain
        \[
         a^L = \frac{1 - a^{L+1}}{3 - 2(a^{L+1} + b^{L+1})}.
        \]
        Similarly, from the second line of~\eqref{eq:system0}, we get
        \[
         b^L = \frac{c^{L+1}}{3 - 2(a^{L+1} + b^{L+1})}  = \frac{\frac{1}{2} - a^{L+1}}{3 - 2(a^{L+1} + b^{L+1})}.
        \]
        We summarize this in the following lemma:
        
        \begin{lemma}
        \label{lemma:h0}
         If $\tau_L = 0$, we have 
         \[
          \bigl( a^L, b^L \bigr) = h_0  \bigl( a^{L+1}, b^{L+1} \bigr),
         \]
         where $h_0$ is the homography transformation defined by
         \[
          h_0(a, b) := \frac{1}{3 - 2(a + b)} \left(1 - a, \frac{1}{2} - a\right).
         \]
        \end{lemma}
        
        \medskip
        
        Now, let us consider the alternative case, where $\tau_L = 1$. Then the analog of~\eqref{eq:system0} becomes
        \begin{equation}
         \label{eq:system1}
         \begin{cases}
          \left| w_k^L \right|_{\A} = \left| w_k^{L+1} \right|_{\C} + O(1) \\
          \left| w_k^L \right|_{\B} = \left| w_k^{L+1} \right|_{\A} + 2 \left| w_k^{L+1} \right|_{\C} + O(1) \\
          \left| w_k^L \right|_{\C} = \left| w_k^{L+1} \right|_{\B} + 2 \left| w_k^{L+1} \right|_{\D} + O(1) \\
          \left| w_k^L \right|_{\D} = \left| w_k^{L+1} \right|_{\D} + O(1)
         \end{cases}
        \end{equation}
        (Observe that the two first lines have been exchanged.) The asymptotics~\eqref{eq:rapport} remains unchanged, and similar arguments lead to the following counterpart of Lemma~\ref{lemma:h0}: 
        
        \begin{lemma}
        \label{lemma:h1}
         If $\tau_L = 1$, we have 
         \[
          \bigl( a^L, b^L \bigr) = h_1  \bigl( a^{L+1}, b^{L+1} \bigr),
         \]
         where $h_1$ is the homography transformation defined by
         \[
          h_1(a, b) := \frac{1}{3 - 2(a + b)} \left( \frac{1}{2} - a, 1 - a\right).
         \]
        \end{lemma}
        
        \subsubsection{Contracting homographies}
        \label{sec:homographies}
        The purpose of this section is to establish nice contracting properties for the homographies $h_0$ and $h_1$ defined in Lemmas~\ref{lemma:h0} and~\ref{lemma:h1}. We first observe that the quantities $a^L$ and $b^L$ are always in $[0,\frac{1}{2}]$, therefore it is natural to consider the action of $h_0$ and $h_1$ on the square $K := [0,\frac{1}{2}] \times [0,\frac{1}{2}]$. Moreover, elementary computations show that for $i\in\{0,1\}$, $h_i(K) \subset E$, where $E$ is the compact domain
        \begin{equation}\label{eq:defE}
           E := \left\{(a, b) \in K \mid a + b \le \frac{3}{4}\right\}
        \end{equation}
        (see Figure~\ref{fig:h0h1}).
        \begin{figure}
         \begin{center}
                \includegraphics[width = 5cm]{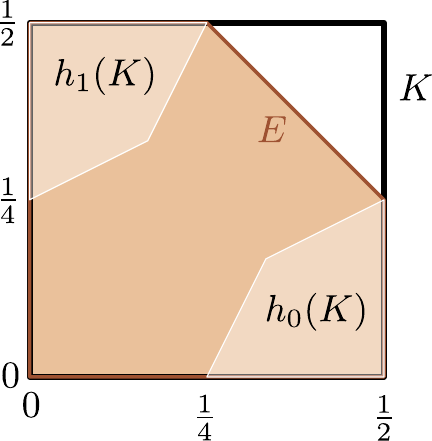}                                           
         \end{center}
        \caption{Homographies $h_0$ and $h_1$ acting on the square $K = [0,\frac{1}{2}] \times [0,\frac{1}{2}]$}\label{fig:h0h1}
        \end{figure} 
        
        It follows by Lemmas~\ref{lemma:h0} and~\ref{lemma:h1} that for every $L \ge 0$, $(a^L, b^L) \in E$. Thus we are led to consider $h_0$ and $h_1$ as maps: $E\to E$.
        
        \begin{lemma}
        \label{lemma:contraction}
         For $i\in\{0,1\}$, $h_i:E\to E$ is contracting for $\|\cdot\|_\infty$. More precisely, for all $p, p'\in E$, we have 
         \[
           \| h_i(p) - h_i(p') \|_\infty \le \frac{2}{3} \| p - p' \|_\infty.
         \]
        \end{lemma}

        \begin{proof}
         It is obviously enough to prove the result for $h_0$. Let $(a, b) \in E$, and let J be the Jacobian of $h_0$ at $(a, b)$
         \[
            J := J_{h_0}(a,b) = \frac{1}{\bigl(3 - 2(a + b)\bigr)^2} 
              \begin{pmatrix}
               2b - 1 & 2 - 2a \\
               2b - 2 & 1 - 2a 
              \end{pmatrix}.
         \]
         For $r \ge 0$, let 
         \[
          B_\infty(r) := \bigl\{(x, y) \in \RR^2 \mid \| (x, y) \|_\infty \le r \bigr\}.
         \]
         We want to prove that $J\bigl(B_\infty(1)\bigr) \subset B_\infty\left(\frac{2}{3}\right)$. Since every point in $B_\infty(1)$ is a convex combination of $(1, 1)$, $(1, -1)$, $(-1, -1)$ and $(-1, 1)$, it is enough to check that
         \begin{equation}
          \label{eq:contraction}
          \left\| J \begin{pmatrix}
                     1 \\
                     -1
                    \end{pmatrix}
                    \right\|_\infty \le \frac{2}{3},
                    \quad\text{ and }\quad
          \left\| J \begin{pmatrix}
                     1 \\
                     1
                    \end{pmatrix}
                    \right\|_\infty \le \frac{2}{3}.
         \end{equation}
        We have 
        \[
         J \begin{pmatrix}
                     1 \\
                     -1
                    \end{pmatrix}
                    = \frac{1}{3 - 2(a + b)}\begin{pmatrix}
                     -1 \\
                     -1
                    \end{pmatrix},
        \]
        and on $E$ we have $3 - 2(a + b) \ge \frac{3}{2}$, which gives the first inequality in~\eqref{eq:contraction}. For the second one, we have
        \[
         J \begin{pmatrix}
                     1 \\
                     1
                    \end{pmatrix}
                    = \frac{1}{\bigl(3 - 2(a + b)\bigr)^2}\begin{pmatrix}
                     1 + 2(b - a) \\
                     -1 + 2(b - a)
                    \end{pmatrix}.
        \]
        But as $(a, b) \in E$, we have
        \[
         |1 + 2(b - a)| = 1 + 2(b - a) \le 2 - 2a \le 3 - 2(a + b), 
        \]
        and similarly         
        \[
         |-1 + 2(b - a)| = 1 + 2(a - b) \le 3 - 2(a + b).
        \]
        This gives 
        \[ \left\| J \begin{pmatrix}
                     1 \\
                     1
                    \end{pmatrix}
                    \right\|_\infty \le \frac{1}{3 - 2(a + b)}
        \] 
        and we conclude the proof of the second inequality in~\eqref{eq:contraction} as the first one.
        \end{proof}
        
        By compacity of $E$, we get the following result:
        
        \begin{corollary}
         For every sequence $t = (t_L)_{L \ge 0} \in \{0, 1\}^\NN$, there exists $\bigl(a(t), b(t)\bigr) \in E$ such that 
         \[ \bigcap_{L \ge 0} h_{t_0} \circ \cdots \circ h_{t_L} (E) = \bigl\{\bigl(a(t), b(t)\bigr)\bigr\}. \]
        \end{corollary}\label{cor:def_a_b}

        Coming back to the frequencies of letters observed in $\Delta^L y$ along the increasing sequence $\bigl(\Sigma_y^L(n_k)\bigr)_{k\ge1}$, we can conclude using Lemmas~\ref{lemma:h0} and~\ref{lemma:h1} together with~\eqref{eq:undemi} that these frequencies are neither dependent on the choice of $(n_k)$ nor on $y$, they depend only on $\tau$:
        
        \begin{corollary}
        \label{cor:symbol_frequency}
         For each $L\ge 0$, we have 
         \[
          a^L = a({S^L \tau}), \quad
          b^L = b({S^L \tau}), \quad
          c^L = \frac{1}{2} - a({S^L \tau}), \quad \text{ and }\quad
          d^L = \frac{1}{2} - b({S^L \tau}).
         \]
        \end{corollary}

        \begin{remark}
        \label{rem:c+d}
        Of course, any limit point $\bigl(a(t), b(t)\bigr)$ has to be in the compact set $E$ defined in~\eqref{eq:defE}. Therefore, for all $L \ge 0$ we have $a^L + b^L \le 3/4$ and $c^L + d^L \ge 1/4$. 
        \end{remark}
        
        \subsubsection{The case of $B_{Y_\tau}^L$}
        
        To use Lemma \ref{lemma:DeltaLygeneric} for general finite words, we need to ensure that for any $L \geq 1$, the measure $\nu(B_{Y_\tau}^L)$ depends only on $\tau$. 
        
        \begin{proposition}
            \label{prop:BLY}
            For any $L \geq 1$, we have
            \[
            \nu\bigl(B_{Y_\tau}^L\bigr) = \lim_{k\to\infty} \frac{\left|w_k^L\right|}{n_k} = \prod_{\ell = 1}^L \frac{1}{1+2(c^\ell + d^\ell)} \tend{L}{\infty} 0.
            \]
        \end{proposition}
        
        \begin{proof}
            Recall that $\left|w_k^L\right| = \Sigma^L_y(n_k)$, so that the first equality is just a reformulation of~(\ref{eq:fconstant1}). Moreover, the convergence to 0 of the product expression as $L \to \infty$ is a straightforward consequence of Remark~\ref{rem:c+d}. So we only need to show the second equality.
            
            Recall, as we already observed, that for all $L \ge 0$ the word $w_k^L$ differs from $\varphi_{\tau_L}(w_k^{L+1})$ by at most two symbols at the beginning and at the end. Therefore, we have
            \begin{align*}
                |w_k^L| & = |\varphi_{\tau_L}(w_k^{L+1})| + O(1) \\
                & = |w_k^{L+1}|_\A + |w_k^{L+1}|_\B + 3|w_k^{L+1}|_\C + 3|w_k^{L+1}|_\D + O(1) \\
                & = |w_k^{L+1}| + 2\bigl(|w_k^{L+1}|_\C + |w_k^{L+1}|_\D\bigr) + O(1) \\
                & = |w_k^{L+1}| \bigl( 1 + 2 (c^{L+1} + d^{L+1}) + o(1)\bigr).
            \end{align*}
            In particular, for $L = 0$ we get 
            \[ n_k = |w_k^1| \bigl( 1 + 2 (c^{1} + d^{1}) + o(1)\bigr), \]
            which yields
            \[ \lim_{k\to\infty} \frac{|w_1^k|}{n_k} = \frac{1}{1 + 2(c^1 + d^1)} \]
            Furthermore, we have in general 
            \[
            \frac{|w_k^{L+1}|}{n_k} = \frac{|w_k^{L+1}|}{n_k} \frac{1}{1 + 2 (c^{L+1} + d^{L+1}) + o(1)}\, ,
            \]
            from which we get the announced formula by an easy induction on $L$. 
        \end{proof}
        
        Now that we know that $\nu(B_{Y_\tau}^L)$ only depends on $\tau$ and not on $(n_k)_{k \geq 0}$, we can look at the frequency of larger words.
        
        \subsubsection{Frequencies of general finite words}
        
        We fix a finite word $w \in \{\A,\B,\C,\D\}^m$, of length $m \in \mathbb{N}$. To prove that the limit of $\frac{|w_k^0|_w}{n_k}$ only depends on $\tau$, the idea is to decompose $w_k$ into images of letters from some $w_k^L$ with large enough $L$.
        
        We define, for all symbol $\S\in \{\A,\B,\C,\D\}$ and all $L \geq 1$, the block $\S_L := \varphi_{\tau[0,L)}(\S)$ (we will also use the notation $s^L$ for the frequency of $\S$ in $\Delta^L y$).
        For example, if $\tau[0,3) = 000$, we have $\B_3 := \varphi_0^3(\B) = \D\C\D\A\B\A\D\C\D$ and $\C_3 := \varphi_0^3(\C) = \A\D\C\D\A$.
        
        \begin{figure}[h!]       
        \begin{center}
        \includegraphics{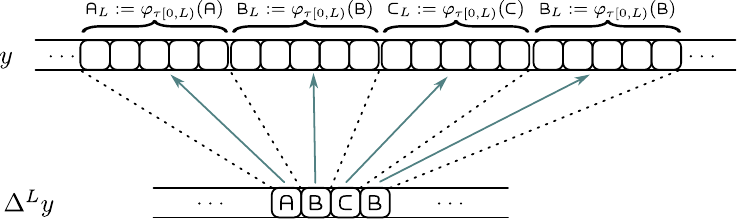}
        \caption{Decomposing a recoding $y$ in terms of images of letters of one of its derivative $\Delta^L y$}
        \end{center}
        \end{figure}
        
        The larger the degree $L$ of derivation, the less likely occurrences of $w$ in $y$ will be to straddle two adjacent $\S_L$. So, as it will be justified in the following lemma, we can approximate $|w_k^0|_w$ by $\sum_{\S\in \{\A,\B,\C,\D\}} |w_k^L|_\S |\S_L|_w$. Dividing this approximation by the length $n_k$ of $w_k^L$, we get
        $$\sum_{\S\in \{\A,\B,\C,\D\}} \frac{|w_k^L|_\S}{n_k}  |\S_L|_w = \frac{\Sigma_y^L(n_k)}{n_k} \sum_{\S\in \{\A,\B,\C,\D\}} \frac{|w_k^L|_\S}{\Sigma_y^L(n_k)}  |\S_L|_w.$$
        By Corollary~\ref{cor:symbol_frequency} and Proposition~\ref{prop:BLY}, the latter expression converges as $k \to \infty$ for any $L \geq 1$, and its limit
        \[ \alpha_L := \nu(B_{Y_\tau}^L) \sum_{\S\in \{\A,\B,\C,\D\}} s^L |\S_L|_w \]
        depends only on $\tau$. 
        
        
        \begin{lemma}
            \label{lemma:freq-general}
            The sequence $(\alpha_L)$ being defined from the word $w$ as above,  for any $\varepsilon > 0$, there exists $L \geq 1$ and $K \geq 1$ such that for all $k \ge K$ 
            $$ \left|\frac{|w_k^0|_w}{n_k} - \alpha_L\right| < \varepsilon.$$
        \end{lemma}
        
        \begin{proof}
            Given $\varepsilon > 0$, by~Proposition~\ref{prop:BLY} we can find $L \geq 1$ such that $\nu(B_{Y_\tau}^L) < \frac{\varepsilon}{4m}$ (recall that $m$ is the length of $w$). 
            For such an $L$, we decompose $y$ into blocks $\S_L$ with $\S \in \{\A,\B,\C,\D\}$.
            We then link the number of occurrences of $w$ in $y$ to the number of occurrences of the respective letters in $\Delta^L y$: for any $k \ge 1$, we have
            \begin{equation}
                \label{eq:decomp-3.10}
                \occu{y}{[1,n_k]}{w}= \sum_{\S \in \{\A,\B,\C,\D\}} \occu{\Delta^L y}{\left[1,\Sigma^L_y(n_k)\right]}{\S} \left|\S_L\right|_w + R_{L,n_k}(w) + R'_{L,n_k}(w),
            \end{equation}
            where the two remaining terms are  
            \begin{itemize}
                \item $R_{L,n_k}(w)$: the number of occurrences of $w$ that are lost due to the potential shift from Proposition~\ref{Prop prim}. We know that $R_{L,n_k}(w) = \mathcal{O}(3^L)$, hence $\frac{R_{L,n_k}(w)}{n_k} \tend{k}{\infty} 0$.
                \item $R'_{L,n_k}(w)$: the number of occurrences of $w$ in $y$ that straddle two adjacent blocks $\S_L$. Since there are at most $\Sigma^L_y(n_k)$ such patterns, there are at most $m  \Sigma^L_y(n_k)$ such occurrences (the factor $m$ comes from the different shifts of $w$ that may straddle two given adjacent blocks). Therefore, we have $\frac{R'_{L,n_k}(w)}{n_k} \leq m  \frac{\Sigma_y^L(n_k)}{n_k}$. By Proposition~\ref{prop:BLY} and by choice of $L$, we get 
                \[ \limsup_{k \to \infty} \frac{R'_{L,n_k}(w)}{n_k} \leq m  \nu(B^L_{Y_\tau}) < \frac{\varepsilon}{4}. \]
            \end{itemize}
            
            Now, dividing the right-hand side of (\ref{eq:decomp-3.10}) by $n_k$,  we get 
            \begin{equation}
               \label{eq:decomp3.10-bis} 
               \frac{\Sigma^L_y(n_k)}{n_k}  \sum_{\S \in \{\A,\B,\C,\D\}} \frac{\occu{\Delta^L y}{\left[1,\Sigma^L_y(n_k)\right]}{\S}}{\Sigma^L_y(n_k)}  \left|\S_L\right|_w
               + 
               \frac{R_{L,n_k}(w)}{n_k}
               + 
               \frac{R'_{L,n_k}(w)}{n_k}.
            \end{equation}
            The first of the three terms above converges to $\alpha_L$ as $k \to \infty$, let $K_1$ be the index from which this first term is close to $\alpha_L$ to within $\varepsilon/3$; let also $K_2$ be the index from which $\frac{1}{n_k}R_{L,n_k}(w) < \varepsilon/3$, and let $K_3$ be the index from which $\frac{1}{n_k}R'_{L,n_k}(w) < \varepsilon/3$. Then we can set $K := \max\{K_1, K_2, K_3\}$, and we have 
            \[ \forall k \geq K,\; \left|\frac{\occu{y}{[1,n_k]}{w}}{n_k} - \alpha_L \right| < \varepsilon. \]
        \end{proof}
        
        \begin{proof}[Proof of Proposition~\ref{prop:accumulation}]
        From the preceding lemma, we can deduce that for any $w \in \{\A,\B,\C,\D\}^*$, the sequence $\left(\frac{1}{n_k}\occu{y}{[1,n_k]}{w}\right)_{k \geq 0}$ is Cauchy, therefore it converges in $[0,1]$ as $k\to\infty$. Furthermore, we can also deduce that the sequence $(\alpha_L)_{L \geq 1}$ also converges as $L\to\infty$ to the same limit as the previous sequence.
        But since every $\alpha_L$ only depends on $w$, $\tau$ (and obviously on $L$), but not on $(n_k)$ nor on $y$, we can conclude that the limit $\nu([w])$ of $\left(\frac{1}{n_k}\occu{y}{[1,n_k]}{w}\right)_{k \geq 0}$  depends only on $\tau$ and $w$. 
        \end{proof}

        This ends the proof of Proposition~\ref{prop:accumulation}. But any $S$-invariant probability measure is completely determined by its values on the cylinder sets $[w]$, $w\in\{\A,\B,\C,\D\}^*$. Therefore, this proves the existence of a unique probability measure $\nu_\tau$ (depending only on $\tau$) such that 
        \begin{equation}
         \label{eq:cv_to_nutau}
         \forall y\in Y_\tau,\ \frac{1}{n} \sum_{j = 1}^n \delta_{S^j y} \tend{n}{\infty} \nu_\tau.
        \end{equation}
        This proves that $\nu_\tau$ is the unique $S$-invariant probability measure on $Y_\tau$ (see Section~\ref{sec:basic}).
        
    \subsection{Unique ergodicity in $X_\tau$}
    \label{sec:UE_X}
    
    Now we want to see how the unique ergodicity of $(Y_\tau,S)$ can be transferred to the corresponding subshift $(X_\tau,S)$ of smooth sequences of type $\tau$. For this, we first introduce the subset of well-aligned bi-infinite smooth sequences of type $\tau$
    \[
     B_{X_\tau} := \{x \in X_\tau \mid x_0 \neq x_{-1}\}.
    \]
    Since mono-symbol blocs of $x\in X_\tau$ have length 1 or 3, it is clear that for all $x\in B_{X_\tau}$, the \emph{return time} to $B_{X_\tau}$
    \[
       r_{B_{X_\tau}}(x) := \min \{ j\ge 1 \mid S^j x \in B_{X_\tau}\}
    \]
    is finite (in this precise case, it takes values $1$ or $3$). This allows us to consider the induced transformation $S_{B_{X_\tau}}: B_{X_\tau} \to B_{X_\tau}$, defined by
    \[
     \forall x\in B_{X_\tau},\ S_{B_{X_\tau}}(x) := S^{r_{B_{X_\tau}}(x)} (x).
    \]
    This induced transformation is a homeomorphism of $B_{X_\tau}$, and thus we have defined the \emph{induced dynamical system} $\bigl( B_{X_\tau}, S_{B_{X_\tau}} \bigr)$. Now, it is straightforward to see that the recoding map $\rec$ is a conjugacy between this induced system  and $(Y_\tau,S)$. As unique ergodicity is a conjugacy invariant, it follows immediately that the induced system $\bigl( B_{X_\tau}, S_{B_{X_\tau}} \bigr)$ is itself uniquely ergodic. Let us denote by $\mu_{B, \tau}$ the unique invariant probability measure on this induced system. 
    
    We now want to explain how the unique ergodicity of $(X_\tau, S)$ follows. Let $\mu$ be any $S$-invariant probability measure on $(X_\tau, S)$. We can see $B_{X_\tau}$ as the disjoint union $B_1 \sqcup B_3$ where $B_1$ (respectively $B_3$) is the set of well-aligned sequences of type $\tau$ whose return-time to $B_{X_\tau}$ is $1$ (respectively $3$). Since any smooth sequence of type $\tau$ is either well-aligned, or the shift of a well-aligned smooth sequence by 1 or 2, we can partition $X_\tau$ as follows:
     \begin{equation}
        \label{eq:partition}
        X_\tau = B_1 \sqcup B_3 \sqcup S B_3 \sqcup S^2{B_3}.
    \end{equation}
    We deduce from~\eqref{eq:partition} and by the $S$-invariance of $\mu$ that for all Borel subset $E\subset X_\tau$, we have 
    \begin{equation}
     \label{eq:mu}
    \begin{aligned}
     \mu(E) &= \mu(E\cap B_1) + \mu(E\cap B_3) + \mu(E\cap SB_3) + \mu(E\cap S^2B_3) \\
     &= \mu(E\cap B_1) + \mu(E\cap B_3) + \mu(S^{-1}E\cap B_3) + \mu(S^{-2}E\cap B_3). 
    \end{aligned} 
    \end{equation}

    In particular, we must have $\mu(B_{X_\tau}) > 0$ (otherwise we would have $\mu(X_\tau) = 0$ by the above formula). So we can consider the conditional probability measure $\mu_{B_{X_\tau}}$ defined, for all Borel subset $E\subset X_\tau$, by 
    \[
     \mu_{B_{X_\tau}}(E) := \mu(E\cap B_{X_\tau}) / \mu(B_{X_\tau}).
    \]
    It is classical (see \emph{e.g.} \cite[Section~2.3]{petersen1983}) that this conditional measure is invariant by the transformation induced on $B_{X_\tau}$. By unique ergodicity of the induced system, we must have $\mu_{B_{X_\tau}} = \mu_{B, \tau}$. Turning back to~\eqref{eq:mu}, we therefore have for all $E$
    \[
      \frac{\mu(E)}{\mu(B_{X_\tau})} =  \mu_{B, \tau}(E\cap B_1) + \mu_{B, \tau}(E\cap B_3) + \mu_{B, \tau}(S^{-1}E\cap B_3) + \mu_{B, \tau}(S^{-2}E\cap B_3).
    \]
    In particular, for $E = X_\tau$, we get 
    \[
     1 / \mu(B_{X_\tau}) = \mu_{B, \tau}(B_1) + 3 \mu_{B, \tau}(B_3),
    \]
    (which is just a particular case of Kac's lemma). Finally, we see that for all $E$, $\mu(E)$ is completely determined by values of the measure $\mu_{B, \tau}$, which proves the unique ergodicity of $(X_\tau, S)$. From now on, we denote by $\mu_\tau$ the unique $S$-invariant probability measure on $X_\tau$. 

     \subsection{A remark on unilateral smooth sequences, the column map $\Phi$ and reversal invariance}

    From the unique ergodicity of any $(X_\tau,S)$, we can conclude that any bi-infinite smooth sequence $x$ admits a frequency for any pattern $w$, which is given by $\mu_\tau([w])$ with $\tau = Types(x)$. Note that, in~\eqref{eq:def_freq}, we have defined  this frequency using only the positive side of $x$, but by unique ergodicity we also have 
    \[
        \mu_\tau([w]) = \overleftarrow{\mathrm{freq}}_x (w) := \lim_{n \to \infty} \frac{|x_{-n} \cdots x_{-1}|_w}{n}.
    \]
    
    Smooth sequences are, as presented in the introduction, usually studied as \emph{unilateral} sequences. 
    But those unilateral smooth sequences exactly correspond to the positive part of bi-infinite sequences which are infinitely well-aligned (\textit{i.e.}  in $\bigcap_{L \geq 1} B_X^L$).
    Reducing ourselves to those sequences, the notion of type can be way more easily characterised.
    
    \begin{proposition}
        \label{prop:type-unilat}
        For any given infinitely well-aligned smooth sequence $x \in \bigcap_{L \geq 1}B_X^L$, we have 
        $$\type(x) = \left\{ \begin{array}{cc}
            0 & \textup{if } x_0 = (\Delta x)_0 \\
            1 & \textup{otherwise} 
        \end{array} \right..$$
    \end{proposition} 
    
    \begin{proof}
    This is a direct consequence of Remark~\ref{rem:type_and_parity}: if both $x$ and $\Delta x$ are well-aligned, 0 is the starting index of a mono-symbol block in $x$ as well as in $\Delta x$. Then $\type(x) = 0$ if and only if these blocks are made of the same symbol.
    \end{proof}
    
    For those infinitely well-aligned words, we introduced at the begining the \emph{column sequence} $\Phi(x) := \bigl(\Delta^k(x)_0\bigr)_{k\geq 0}$. With such a characterisation of the type, we can fully obtain $Types(x)$ based on $\Phi(x)$.
    But even more than that, we have the following result:
    
    \begin{proposition}
        \label{prop:swapped-column}        
        For any given infinitely well-aligned $x \in \bigcap_{L \geq 1} B_X^L$, we can construct another, infinitely well-aligned, smooth sequence $\tilde{x} := \Phi^{-1}\left(\overline{\Phi(x)}\right)$ where $\overline{\Phi(x)}$ is the column sequence deduced from $\Phi(x)$ by switching 1s and 3s. 
        Then, $x$ and $\tilde{x}$ have the same frequencies for any pattern.
    \end{proposition}
    
    \begin{proof}
        The proof is immediate as soon as we observe that, by Proposition~\ref{prop:type-unilat}, $x$ and $\tilde{x}$ have the same sequence of types.
    \end{proof}
    
    From a unilateral perspective, those two sequences would seem unrelated. However, in the context of bi-infinite smooth sequences, 
    we see that $x$ and $\tilde{x}$ correspond in fact to the same bilateral sequence, but either read from left to right or the other way around: for all $n \in \ZZ$, we have $\tilde{x}_n = x_{-1-n}$ (see Figure~\ref{fig:reversal}).

 \begin{figure}[ht]
 \centering
 \includegraphics{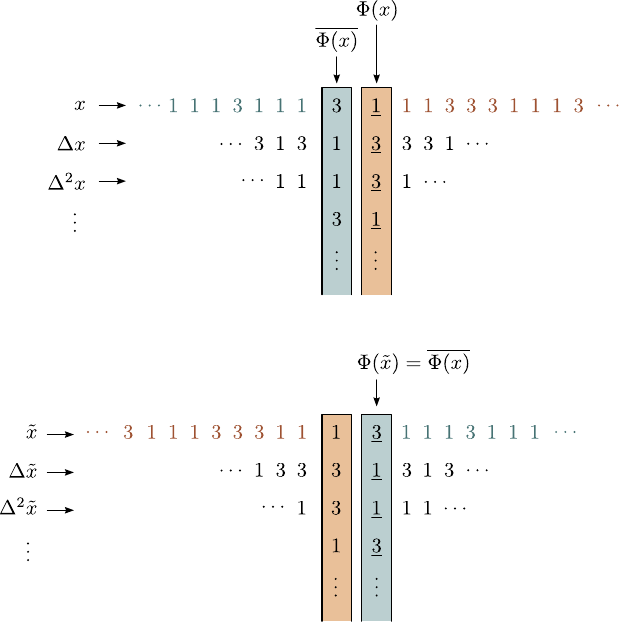}
 \caption{An infinitely well-aligned smooth sequence $x$, its column sequence $\Phi(x)$, and the reversal $\tilde x$ whose column sequence is $\overline{\Phi(x)}$ (that is: it is constructed from $\Phi(x)$ by switching 1s and 3s).}
 \label{fig:reversal}
 \end{figure}
    
    This results in the following corollary on the reversal invariance of the measure $\mu_\tau$:
    
    \begin{corollary}
     \label{cor:reversal}
     For any $\tau \in \{0,1\}^\NN$ and any finite word $w = w_0\dots w_{\ell-1} \in \{0, 1\}^*$, we have 
     \[
     \mu_\tau\left(\left[\overleftarrow{w}\right]\right) = \mu_\tau([w]),     \]
     where $\overleftarrow{w} := w_{\ell-1}\dots w_0$.
    \end{corollary}
    
    \begin{proof}
    Let $x$ be an infinitely well-aligned smooth sequence whose sequence of types is $\tau$, and $\tilde x$ its reversal as described above. (In fact, from Proposition~\ref{prop:type-unilat}, there are exactly two infinitely well-aligned smooth sequence corresponding to the sequence of types $\tau$: $x$ and $\tilde x$.) Then we have the following equalities:
    \[
     \mu_\tau \left(\left[\overleftarrow{w}\right]\right) 
     = \mathrm{freq}_x \left(\overleftarrow{w}\right) 
     = \overleftarrow{\mathrm{freq}}_{\tilde x} \left({w}\right)
     = {\mathrm{freq}}_{\tilde x} \left({w}\right)
     = \mu_\tau([w]).
    \]
    \end{proof}

This result complements Proposition~15 of~\cite{BJP08}, in which the authors prove that the set of factors of any (unilateral) smooth sequence is closed under reversal.    
  
\section{Minimality}
\label{sec:minimal}
  Once we know that $(Y_\tau, S)$ and $(X_\tau, S)$ are uniquely ergodic for all $\tau$, the question of the minimality of these systems arises naturally. We recall the following classical result (see \emph{e.g.} \cite[Theorem~6.17]{Walters1982}): if $(X, T)$ is a uniquely ergodic topological dynamical system, with unique invariant probability measure $\mu$, then $(X, T)$ is minimal if and only if $\mu(U) > 0$ for all non-empty open set $U\subset X$. In the context of a symbolic dynamical system, this condition is equivalent to saying that the unique invariant measure charges every cylinder $[w]$ for all finite words $w$ in the language of the subshift. We have therefore to study whether, for a given $\tau\in\{0,1\}^\NN$, the following is true: 
  \begin{equation}
   \label{eq:minimality_condition}
   \forall w\in \mathcal{L}(Y_\tau),\ \nu_\tau([w]) > 0.
  \end{equation}

  \subsection{Non minimal cases}
  Our first observation is that this minimality condition does not always hold, as shows the following proposition.
  
  \begin{proposition}
   \label{prop:1111}
   For $\tau = 1^\infty = (1, 1, 1, \ldots)$, $(Y_\tau, S)$ is not minimal.
  \end{proposition}
  
  \begin{proof}
   The argument exploits a crucial property of $\varphi_1$ acting on letters: the only way to get a $\D$ starting from a single letter is when this initial letter is itself a $\D$. Let us consider the following sequence of finite words, whose letters are indexed by increasing subintervals of $\ZZ$ (the underlined term indicates the zero position):
   \begin{align*}
    & \underline{\D} \\
    \C &\underline{\D} \C\\    
    \B \A \B \C &\underline{\D} \C \B \A \B \\
    \C \B \C \B \A \B \C &\underline{\D} \C \B \A \B \C \B \C \\
    &\vdots
   \end{align*}
These words are produced starting from an initial $\D$ at the zero position and iterating the map $S\circ \varphi_1$, so that they converge to a bi-infinite alternating sequence $\bar y \in \mathcal{A}^\ZZ$ satisfying $\bar y = S \circ \varphi_1 (\bar y)$. Moreover, this limit point $\bar y$ also enjoys the following properties:
\begin{itemize}
 \item $\bar y \in Y_{1^\infty}$ (this follows from Proposition~\ref{prop:elements_of_Ytau}) ;
 \item the letter $\D$ appears only once in $\bar y$, at the zero position.
\end{itemize}
It follows that the single-letter word $\D$ is in $\mathcal{L}(Y_{1^\infty})$, but as $\bar y$ is generic for the probability measure $\nu_{1^\infty}$, we must have $\nu_{1^\infty}([\D]) = 0$. This proves that the condition~\eqref{eq:minimality_condition} is not satisfied when $\tau = 1^\infty$.
  \end{proof}

\begin{corollary}
 \label{cor:nonminimal}
 If for some $L \ge 0$ we have $S^L\tau = 1^\infty$ (that is: if $\tau$ contains only finitely many $0$s), then $(Y_\tau, S)$ is not minimal.
\end{corollary}

\begin{proof}
 Let $\tau$ and $L$ be as in the statement of the corollary, and let $\bar y$ be as in the proof of the preceding proposition.  We define
 \[ \tilde y := \varphi_{\tau_0} \circ \cdots \circ \varphi_{\tau_{L-1}} \, \bar y \ \in Y_\tau. \]
 Then $\bar y = \Delta^L(\tilde y)$. Now, as $\bar  y_0 = \D$, by continuity of $\Delta^L$ there exists some $m \ge 1$ such that
 \[
  \forall y \in Y_\tau,\ y_{[-m, m]} = \tilde y_{[-m, m]} \Longrightarrow \left(\Delta^L y\right)_0 = \D.
 \]
 In particular, taking into account the commutation formula, if $w := \tilde y_{[-m, m]}$ appears in $\tilde y$ at position $[n-m, n+m]$ for some $n \ge 0$, we must have 
 \[
  \left(\Delta^L S^n \tilde y\right)_0 = \bar y_{\Sigma_{\tilde y}^L(n)} = \D.
 \]
 Since $\D$ appears only at position $0$ in $\bar y$, this implies that $\Sigma_{\tilde y}^L(n) = 0$, and then that $n \le 3^n$. Therefore by genericity of $\tilde y$ for $\nu_\tau$, we get that  $\nu_\tau([w]) = 0$. But $w \in \mathcal{L}(Y_\tau)$. As in the preceding proof, we can conclude that~\eqref{eq:minimality_condition} is not satisfied for~$\tau$.  
\end{proof}

\subsection{Other cases are minimal}

Now we want to show that, except for the cases discussed in the previous section, the subshift $(Y_\tau, S)$ is minimal. 

\begin{theorem}
 \label{thm:minimal}
 For all $\tau \in \{0, 1\}^\NN$, the subshift $(Y_\tau, S)$ is minimal if and only if $\tau$ contains infinitely many $0$s. 
\end{theorem}

To establish what remains to be proved in this statement, we will show by induction on $n\ge1$ that for all $\tau$ containing infinitely many $0$s and all word $w \in \mathcal{L}(Y_\tau)$ of length $|w| = n$, we have $\nu_\tau([w]) > 0$. For the base case, we have to examine the conditions for the measure $\nu_\tau$ to give positive mass to all single-letter words. 

\begin{lemma}
 \label{lemma:single-letter}
 If $S^2 \tau \neq 1^\infty$, then for all $\S \in \mathcal{A}$ we have $\nu_\tau([\S]) > 0$. 
\end{lemma}

\begin{proof}
 From the analysis developped in Section~\ref{sec:homographies} (see in particular Corollary~\ref{cor:symbol_frequency}), it is enough to prove that when $S^2\tau \neq 1^\infty$, the point $\bigl(a(\tau), b(\tau)\bigr)$ lies in the interior of the square $K = [0, 1/2] \times [0, 1/2]$.
 Let us name the four sides of the square $K$ 
 \begin{align*}
  A &:= \{0\} \times [0, 1/2], \\
  B &:= [0, 1/2] \times \{0\}, \\
  C &:= \{1/2\} \times [0, 1/2], \\
  D &:= [0, 1/2] \times \{1/2\}. 
 \end{align*}  
 
 \begin{figure}[ht]
 \centering
 \includegraphics[width = 14cm]{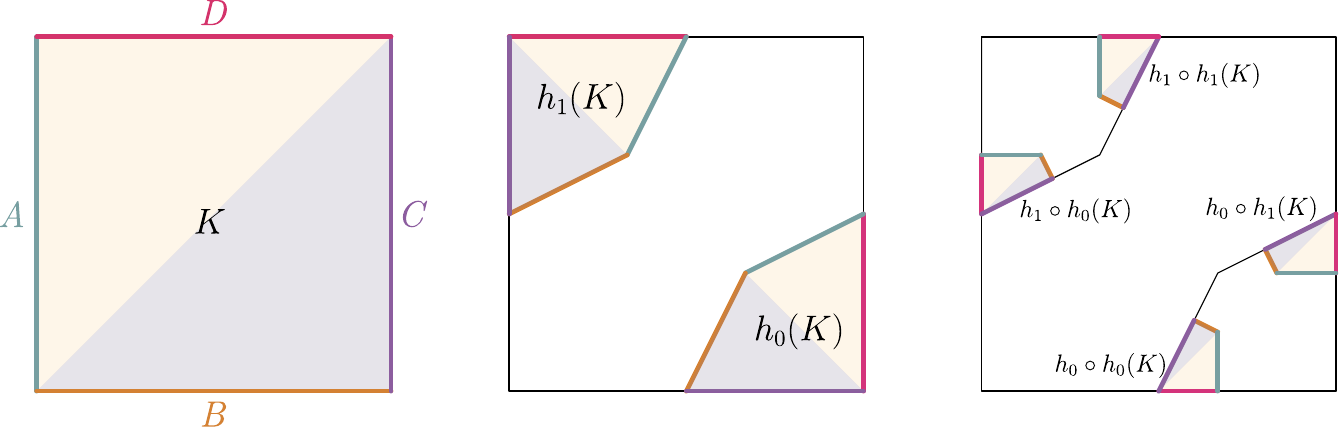}
 \caption{First iterations of the homographies $h_0$ and $h_1$ on the square $K$.}
 \label{fig:iteration}
 \end{figure}
 
 It is straightforward to check that, as indicated on Figure~\ref{fig:iteration}, we have
 \[
    h_0^{-1}(A) = h_0^{-1}(D) = h_1^{-1}(B) = h_1^{-1}(C) = \emptyset,
 \]
and that
\[
  h_0^{-1}(C) = h_1^{-1} (D) = D, \qquad h_0^{-1}(B) = h_1^{-1}(A) = C.
\]

 \begin{figure}[h]
 \centering
 \includegraphics[width = 25mm]{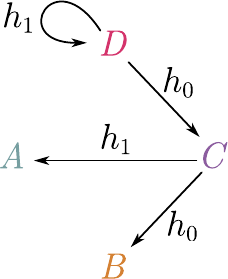}
 \caption{Action of $h_0$ and $h_1$ on the sides of the square}
 \label{fig:graph}
 \end{figure}
 
This action of $h_0$ and $h_1$ on the sides of the square is summarized in the oriented graph on Figure~\ref{fig:graph}. Observe also that for all $\tau\in\{0,1\}^\NN$ we have
\[
   \bigl(a(\tau), b(\tau)\bigr) = h_{\tau_0} \bigl(a(S\tau), b(S\tau)\bigr).
\]
Then an easy exercise which is left to the reader successively gives the following facts:
\begin{enumerate}
 \item $\bigl(a(\tau), b(\tau)\bigr) \in D$ if and only if $\tau = 1^\infty$, which is the only case where $\nu_\tau([\D]) = 0$. 
 \item $\bigl(a(\tau), b(\tau)\bigr) \in C$ if and only if $\tau = 01^\infty$, which is the only case where $\nu_\tau([\C]) = 0$. 
 \item $\bigl(a(\tau), b(\tau)\bigr) \in A$ if and only if $\tau = 101^\infty$, which is the only case where $\nu_\tau([\A]) = 0$. 
 \item $\bigl(a(\tau), b(\tau)\bigr) \in B$ if and only if $\tau = 001^\infty$, which is the only case where $\nu_\tau([\B]) = 0$. 
\end{enumerate}

 This is enough to conclude the proof. 
\end{proof}

Then, for the induction step in the proof of Theorem~\ref{thm:minimal}, we use the following lemma.

\begin{lemma}
 \label{lemma:induction_step}
 For any $\tau\in\{0, 1\}^\NN$ and any $w\in\mathcal{L}(Y_\tau)$ of length $|w|\ge2$, there exists $L\ge1$ and $w'\in\mathcal{L}(Y_{S^L\tau})$ such that
 \begin{itemize}
  \item $|w'| < |w|$,
  \item $w$ is a factor of $\varphi_{\tau_0} \circ \varphi_{\tau_1} \circ \cdots \circ \varphi_{\tau_{L-1}}(w')$.
 \end{itemize}
\end{lemma}

\begin{proof}
 We first prove the lemma for words $w$ of length 2. Remember that the symbols $\A$, $\B$, $\C$ and $\D$ in $Y$ correspond respectively to the blocks $1$, $3$, $111$ and $333$ in $X$. Since blocks of $1$s and blocks of $3$s alternate in $X$, the only words of length~2 that may appear in $\mathcal{L}(Y)$ are 
 \[ \A\B,\ \B\A,\ \B\C,\ \C\B,\ \C\D,\ \D\C,\ \D\A,\ \A\D. \]
 Moreover, after applying $\varphi_0$, the symbol $\B$ always comes surrounded with $\A$s. Therefore, if $\tau_0 = 0$, the words of length 2 that may appear in $\mathcal{L}(Y_\tau)$ are reduced to  
 \[ \A\B,\ \B\A,\ \C\D,\ \D\C,\ \D\A,\ \A\D. \]
 But then, we observe that all these words appear both as factors of 
 \[ \varphi_0 \circ \varphi_0 (\D) = \D\C\D\A\B\A\D\C\D \]
 and of 
 \[ \varphi_0 \circ \varphi_1 (\D) = \A\B\A\D\C\D\A\B\A. \]
 Similarly, after applying $\varphi_1$, the symbol $\A$ always comes surrounded with $\B$s. Therefore, if $\tau_0 = 1$, the words of length 2 that may appear in $\mathcal{L}(Y_\tau)$ are reduced to  
 \[ \A\B,\ \B\A,\ \B\C,\ \C\B,\ \C\D,\ \D\C. \]
 But all these words appear both as factors of 
 \[ \varphi_1 \circ \varphi_0 (\D) = \C\D\C\B\A\B\C\D\C \]
 and of 
 \[ \varphi_1 \circ \varphi_1 (\D) = \B\A\B\C\D\C\B\A\B. \]
This settles the case of words $w$ of length 2.
 
 Now, assume by contradiction that for some $\tau\in\{0,1\}^\NN$ there exists $w\in\mathcal{L}(Y_\tau)$ with $|w| \ge 3$ and for which the conclusion of the lemma does not hold. 
 Let $y \in Y_\tau$ be such that $w$ occurs as a factor of $y$. 
 Since $y = S^j \varphi_{\tau_0} \circ \varphi_{\tau_1} (\Delta^2 y)$ for some $j$, we can consider a minimal factor $w'$ of $\Delta^2 y$ such that $w$ appears as a factor of $\varphi_{\tau_0} \circ \varphi_{\tau_1}(w')$.
 \begin{figure}[h]
 \centering
 \includegraphics[width = 80mm]{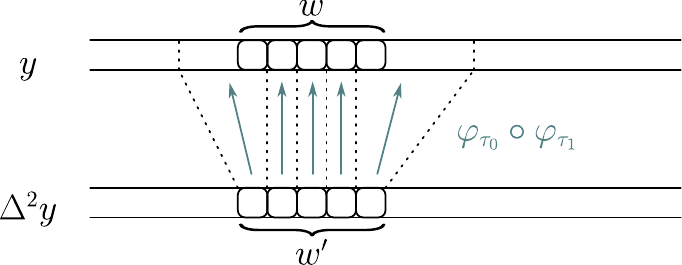}
 \caption{Getting $w$ by substitution from $w'$ with $|w'| \ge |w|$}
 \label{fig:substitution}
 \end{figure}
 By assumption we have $|w'| \ge |w|$, but by minimality of $w'$ the only possibility is that $|w'| = |w|$ (every letter in $w'$ corresponds via $\varphi_{\tau_0}$ to at least one letter in $w$). Let us write $w = w_1\cdots w_k$ and $w' = w'_1\cdots w'_k$, where $k := |w| \ge 3$. Then for all $2 \le i \le k - 1$, $\varphi_{\tau_0} \circ \varphi_{\tau_1} (w'_i)$ has to be reduced to the single letter $w_i$. But, as it is easy to check, there are only 3 ways that for $\S \in \mathcal{A}$, $|\varphi_{\tau_0} \circ \varphi_{\tau_1} (\S)| = 1$:
 \begin{align*}
  \varphi_0 \circ \varphi_0 (\A) &= \A, \\
  \varphi_0 \circ \varphi_1 (\A) &= \D, \\
  \varphi_1 \circ \varphi_1 (\A) &= \C.  
 \end{align*}
 In particular, for all $2 \le i \le k - 1$ we must have $w'_i = \A$. But $\A\A$ never appears in $Y$, therefore $k = 3$. 
 
 Moreover, we observe that $w'$ must enjoy the same property as $w$ (that is: the conclusion of the lemma also fails for $w'$). The same argument as above applies also to $w'$, but since $w'_2 = \A$, only the first of the three cases is possible here. A straightforward induction then proves that $\tau_L = 0$ for all $L \ge 2$, hence $w' \in \mathcal{L}(Y_{0^\infty})$. Yet, it is easy to check that the only words of length 3 with a central $\A$ in $\mathcal{L}(Y_{0^\infty})$ are 
 $\B\A\D$, $\D\A\B$, and $\D\A\D$. But the first two appear in 
 $\varphi_0^2 (\D) = \D\C\D\A\B\A\D\C\D$, and the last one occurs in $\varphi_0^3 (\D) = \D\C\D\A\B\A\D\C\D\A\D\A\D\C\D\A\B\A\D\C\D$. This contradicts the assumption that the lemma fails for $w'$.  
\end{proof}

\begin{proof}[End of the proof of Theorem~\ref{thm:minimal}]
It remains to complete the induction step. Assume that for some $n\ge2$, for all $\tau$ containing infinitely many $0$s, and for all $w \in \mathcal{L}(Y_\tau)$ with $|w| \le n$, we have $\nu_\tau([w]) > 0$. 

We fix $\tau\in\{0,1\}^\NN$ with infinitely many $0$s, and we consider $w \in \mathcal{L}(Y_\tau)$ of length $|w| = n+1$. Applying Lemma~\ref{lemma:induction_step} to $w$, we obtain $L \ge 1$ and $w'\in \mathcal{L}(Y_{S^L\tau})$ such that
\begin{itemize}
  \item $|w'| < |w|$,
  \item $w$ is a factor of $\varphi_{\tau_0} \circ \cdots \circ \varphi_{\tau_{L-1}}(w')$. 
\end{itemize}
We can apply the induction hypothesis to $w'$, which gives $\nu_{S^L\tau}([w']) > 0$. In particular, for any $y'\in Y_{S^L \tau}$, the factor $w'$ appears with positive density in $y'$. Now, let $y$ be any sequence in $Y_\tau$, and set $y' := \Delta^L y \in Y_{S^L \tau}$. We then have $y = S^j \varphi_{\tau_0} \circ \cdots \circ \varphi_{\tau_{L-1}} (y')$ for some $j$, hence $w$ appears with positive density in $y$, and this is enough to conclude that $\nu_\tau([w]) > 0$.
\end{proof}

\section{Perspectives}

    \subsection{Study of the set of possible frequencies} 

    Given a sequence type $\tau$, we recall the notation $a = a(\tau) := \nu_\tau([\A])$, and likewise for $\B,\C$ and $\D$. By the unique ergodicity of $(X_\tau,S)$, all the sequences of $X_{\tau}$ have the same frequency of $1$s. We denote by $f_{\tau} :=\mu_\tau([1])$ the value of this frequency. It follows from the definition of the mono-symbol blocks and from Lemma~\ref{lemma:P7} that $f_{\tau}$ can be computed from the frequencies of letters in $Y_{\tau}$, using the following formula:
    \begin{equation}
        \label{eq : freq1}
        f_{\tau} = \frac{a+3c}{a+b+3(c+d)} = \frac{\frac{3}{2}-2a}{3-2(a+b)}.
    \end{equation}
    Note that any such $a$ and $b$ come from a sequence of iterations of the two homographies $h_0$ and $h_1$, see Corollary~\ref{cor:def_a_b}. 
    Precisely, let $\mathcal{F} := \left\{(a(\tau),b(\tau)) \mid \tau \in \{0,1\}^\mathbb{N}\right\}$. The set $\mathcal{F}$ has a self-similar structure, that can be approximated by computing the images of the square $K=[0,\frac{1}{2}]\times [0,\frac{1}{2}]$ by finite sequences of iterations of the two homographies $h_0$ and $h_1$, as shown on Figure~\ref{fig:f-approx}. 
    Now, let $F := \left\{f_{\tau} \mid \tau \in \{0,1\}^\mathbb{N}\right\}$ be the set of possible values for the frequency of $1$s in smooth sequences.  
    Figure~\ref{fig:fractale-freq} illustrates the successive approximations of $F$ derived from the approximations of $\mathcal F$ of Figure~\ref{fig:f-approx}. Note that $F$ is symmetrical with respect to $1/2$, since swapping the $1$s and the $3$s in a sequence results in moving from a frequency $f$ to a frequency $1-f$.
    
    \begin{figure}[!h]%
        \centering
        \subfloat[\centering Rank $1$]{{\includegraphics[width=.3\textwidth]{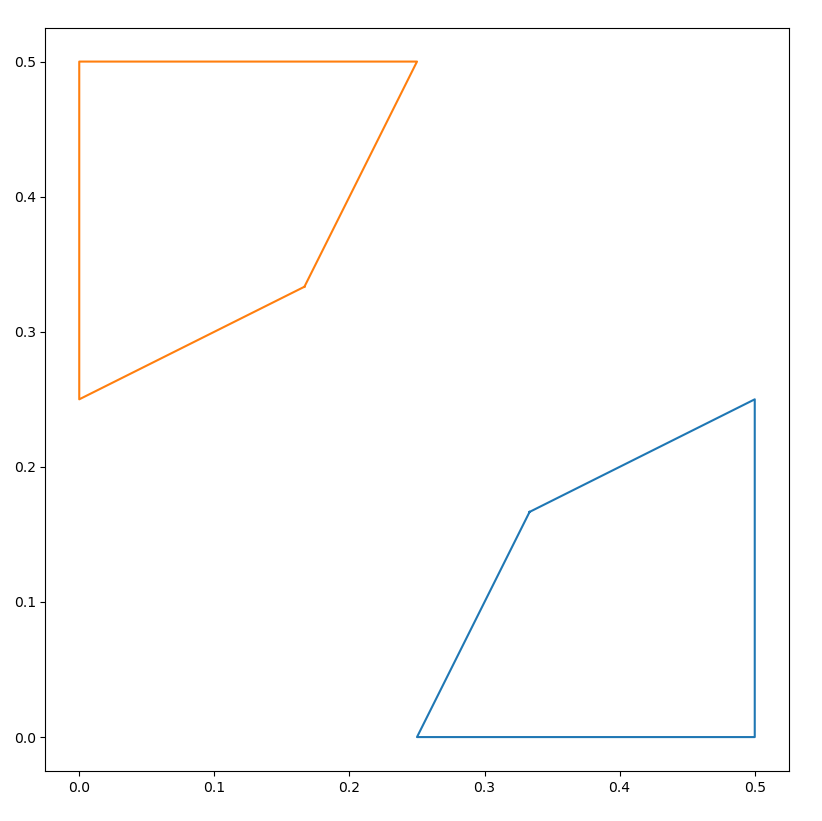}}}%
        \subfloat[Rank $2$]{{\includegraphics[width=.3\textwidth]{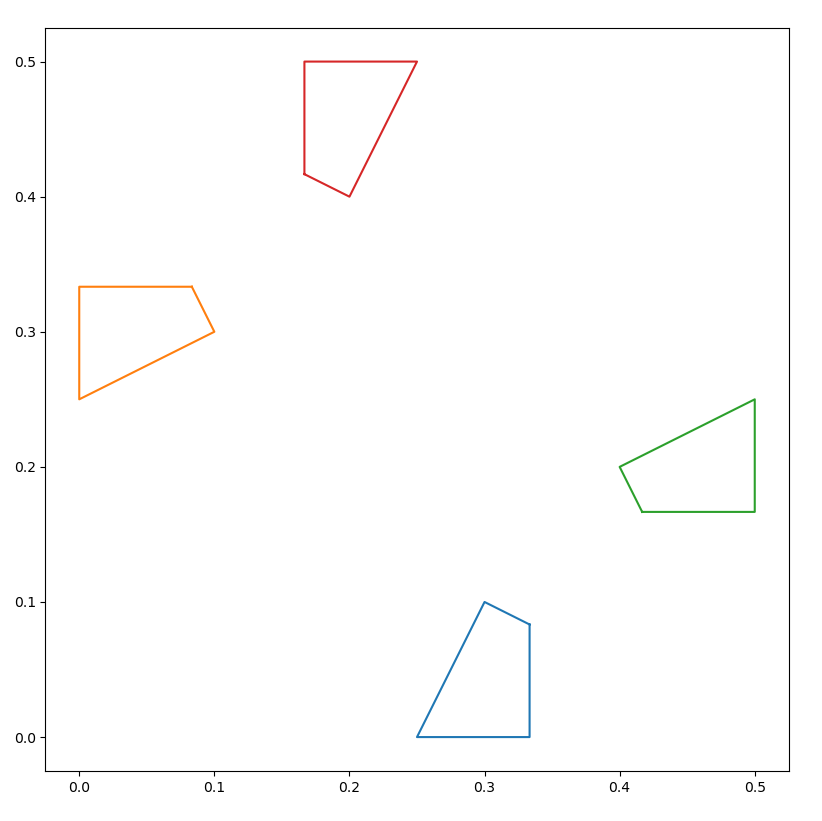}}}
        \subfloat[Rank $3$]{{\includegraphics[width=.3\textwidth]{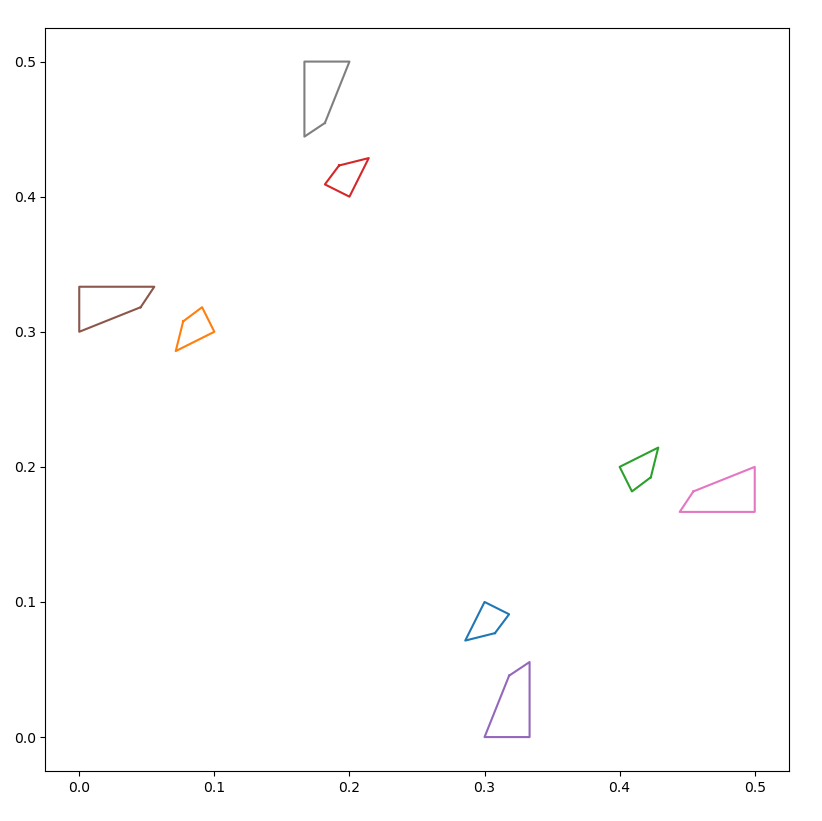}}}
        \\
        \subfloat[\centering Rank $4$]{{\includegraphics[width=.3\textwidth]{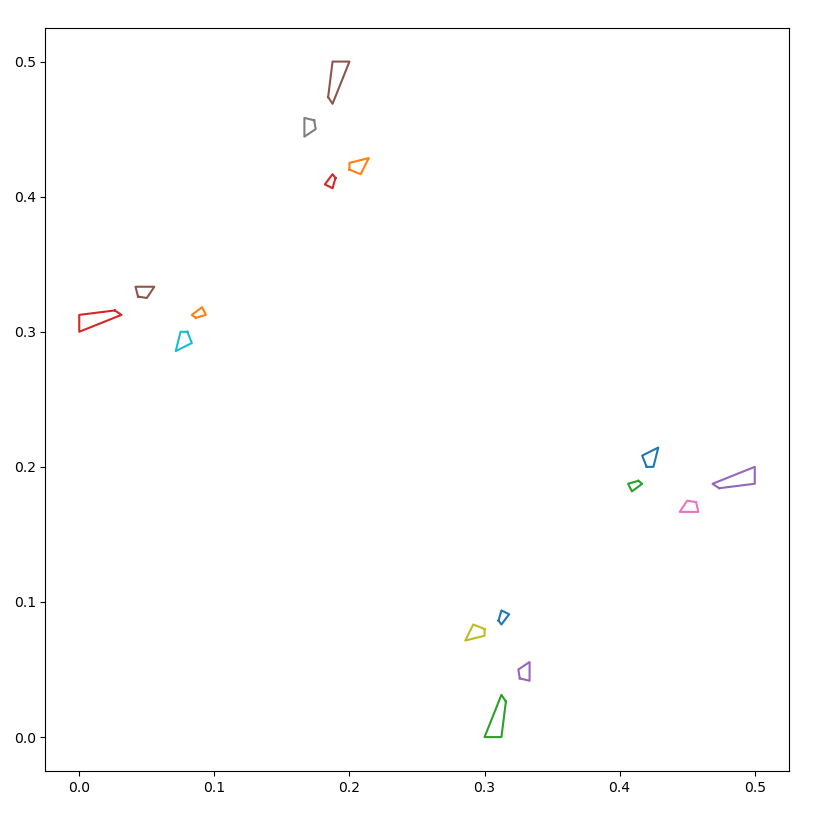}}}%
        \subfloat[Rank $6$]{{\includegraphics[width=.3\textwidth]{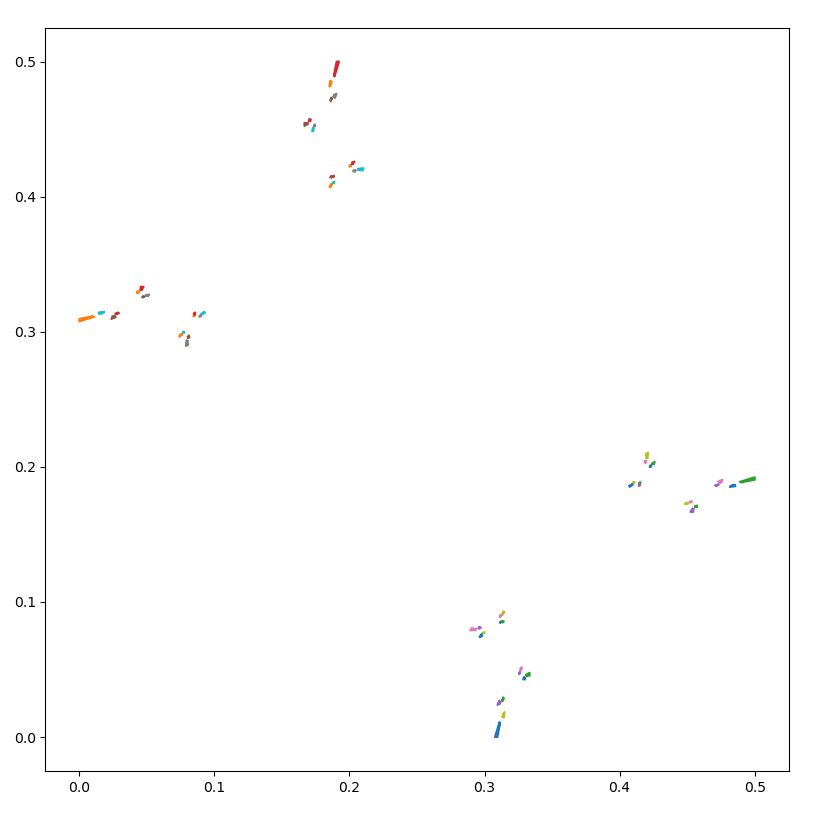}}}
        \subfloat[Rank $8$]{{\includegraphics[width=.3\textwidth]{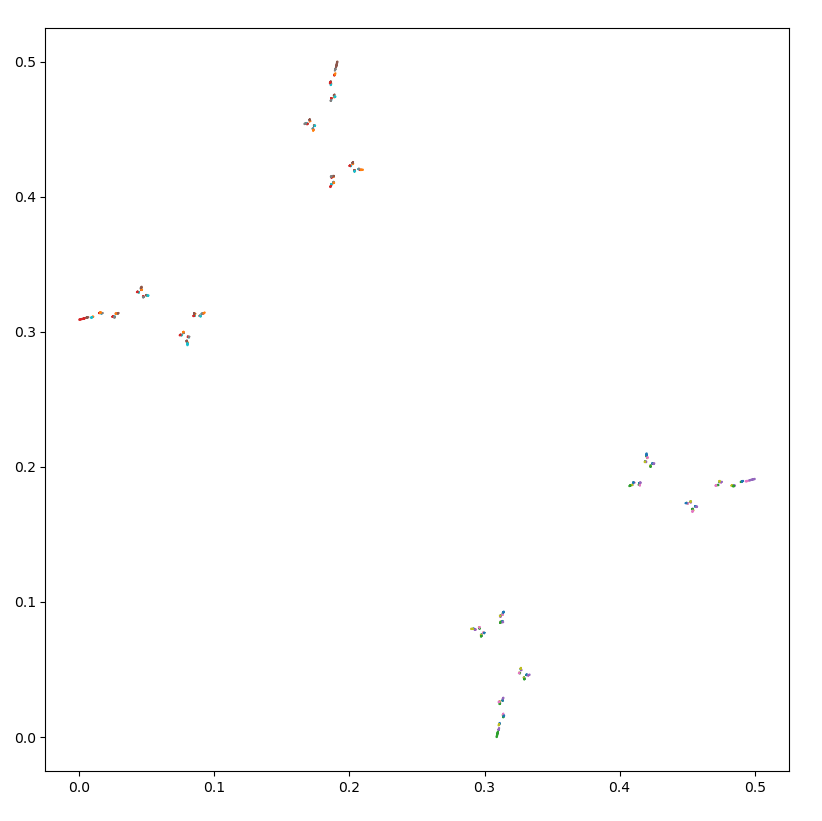}}}
        \caption{Successive approximations of $\mathcal{F}$ obtained by computing the images of the square $K$ by sequences of iterations of the two homographies $h_0$ and $h_1$.}\label{fig:f-approx}
    \end{figure}
    
    \begin{figure}[h!]
        \centering
        \includegraphics[width = \textwidth]{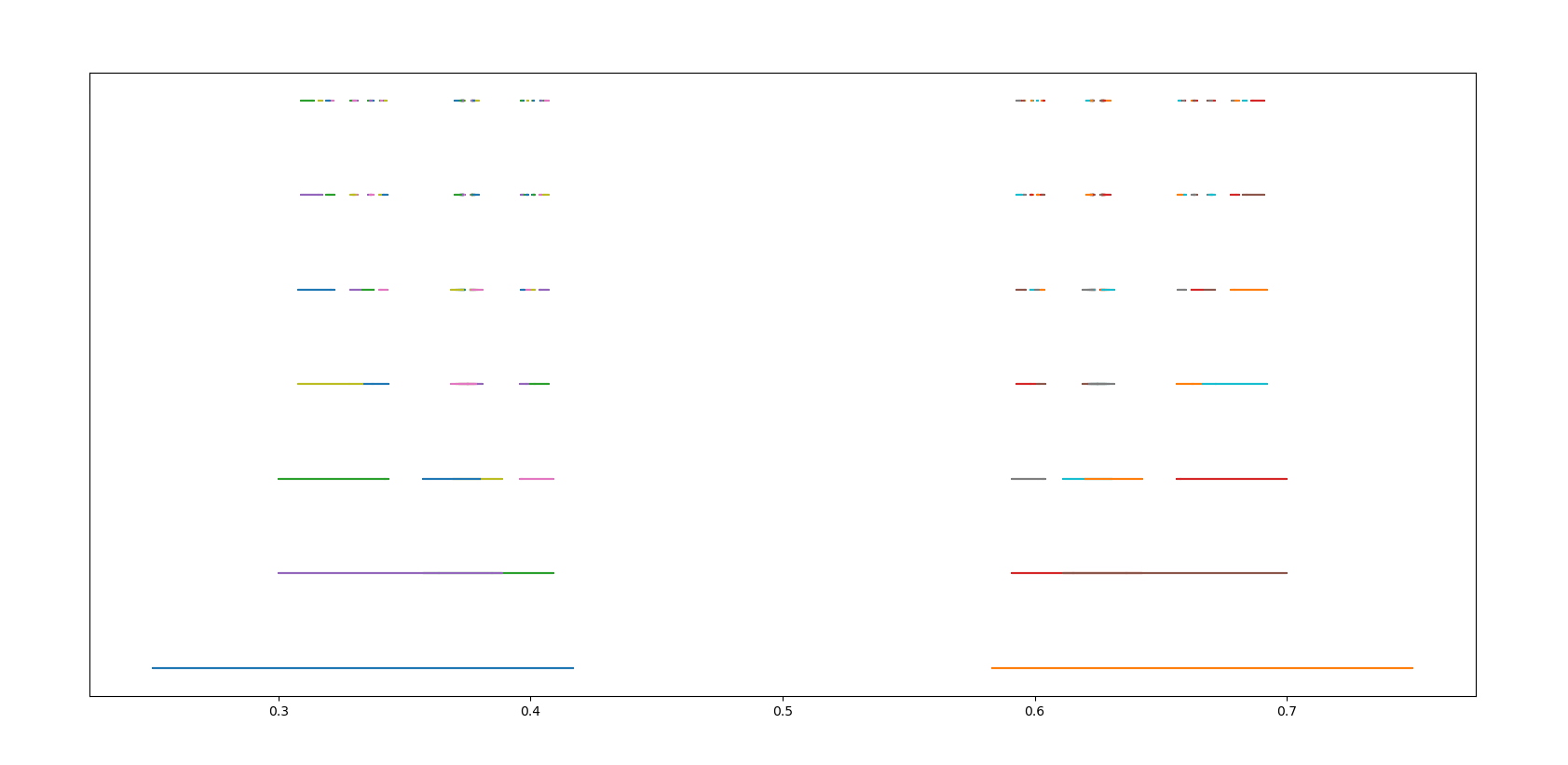}
        \caption{Successive approximations of $F$ (from bottom up) derived from the approximations of $\mathcal F$ using \eqref{eq : freq1}, up to rank $7$.} 
        \label{fig:fractale-freq}
    \end{figure}

    When the sequence $\tau$ is ultimately periodic, the value of $f_{\tau}$ has an explicit description. In particular, for $\tau_1=0001^\infty$, we obtain that $f_{\tau_1}=\frac{1}{2} - \frac{1}{13-\sqrt{5}}$, and for $\tau_2=1^\infty$, $f_{\tau_2}=\frac{1}{2} + \frac{2}{3+\sqrt{5}}$. 
    We conjecture that these sequences achieve the maxima of $F$ on the intervals $[0,1/2]$ and $[0,1]$.
    
    \begin{problem} Examine in further detail the fractal structure of the set $F$ of possible values for the frequency of $1$'s in smooth sequences. Is it true that 
    $\sup F\cap [0,1/2] =\frac{1}{2} - \frac{1}{13-\sqrt{5}}$,
    and 
    $\sup F= \frac{1}{2} + \frac{2}{3+\sqrt{5}}$ ?
    \end{problem}

    \subsection{Generalisation to other alphabets} 
    The bi-infinite smooth sequence and recoding formalisms can be extended to any two-integers alphabet $\{\alpha,\beta\}$ with $\A = \alpha^\alpha, \B = \beta^\alpha, \C = \alpha^\beta$ and $\D = \beta^\beta$. 
    When both $\alpha$ and $\beta$ are odd integers, the properties of Section~\ref{sec:objects_tools} still hold, with the two substitutions becoming:
\[
     \varphi_0:\begin{cases}\A & \mapsto\ \A(\B\A)^p \\
                            \B & \mapsto\ \D(\C\D)^p \\
                            \C & \mapsto\ \A(\B\A)^q \\
                            \D & \mapsto\ \D(\C\D)^q 
               \end{cases}     
      \qquad      
     \varphi_1:\begin{cases}\A & \mapsto\ \B(\A\B)^p \\
                            \B & \mapsto\ \C(\D\C)^p \\
                            \C & \mapsto\ \B(\A\B)^q \\
                            \D & \mapsto\ \C(\D\C)^q 
               \end{cases}.     
    \] 
    with $p = \frac{\alpha-1}{2}$ and $q = \frac{\beta-1}{2}$. 
    The two corresponding homographies are given by
    \[
          h_0(a, b) := \frac{1}{\beta - (\beta-\alpha)(a + b)} \left( \frac{\beta+1}{4} - a\cdot\frac{\beta-\alpha}{2}, \frac{\beta-1}{4} - a\cdot\frac{\beta-\alpha}{2}\right),
         \]
        \[
          h_1(a, b) := \frac{1}{\beta - (\beta-\alpha)(a + b)} \left(\frac{\beta+1}{4} - a\cdot\frac{\beta-\alpha}{2}, \frac{\beta-1}{4} - a\cdot\frac{\beta-\alpha}{2}\right).
         \]

    However, if we want to use the same arguments as in Section~\ref{sec:unique_ergo} to prove the unique ergodicity, we need to ensure that these homographies are contracting on a suitable domain. But this does not seem to be straightforward,  depending on the values of $\alpha$ and $\beta$. Indeed, finite approximations of the analog of the limit set $\mathcal{F}$ suggest that,  for alphabets of the form $\{1,\beta\}$ with $\beta \geq 5$, the contraction property of the homographies might no longer be valid (see Figures~\ref{fig:approx-1,9} and~\ref{fig:approx-1,17}).

    \begin{figure}[!ht]
        \centering
        \subfloat[\centering $\{\alpha,\beta\} = \{1,9\}$\label{fig:approx-1,9}]{{\includegraphics[width=.33\textwidth]{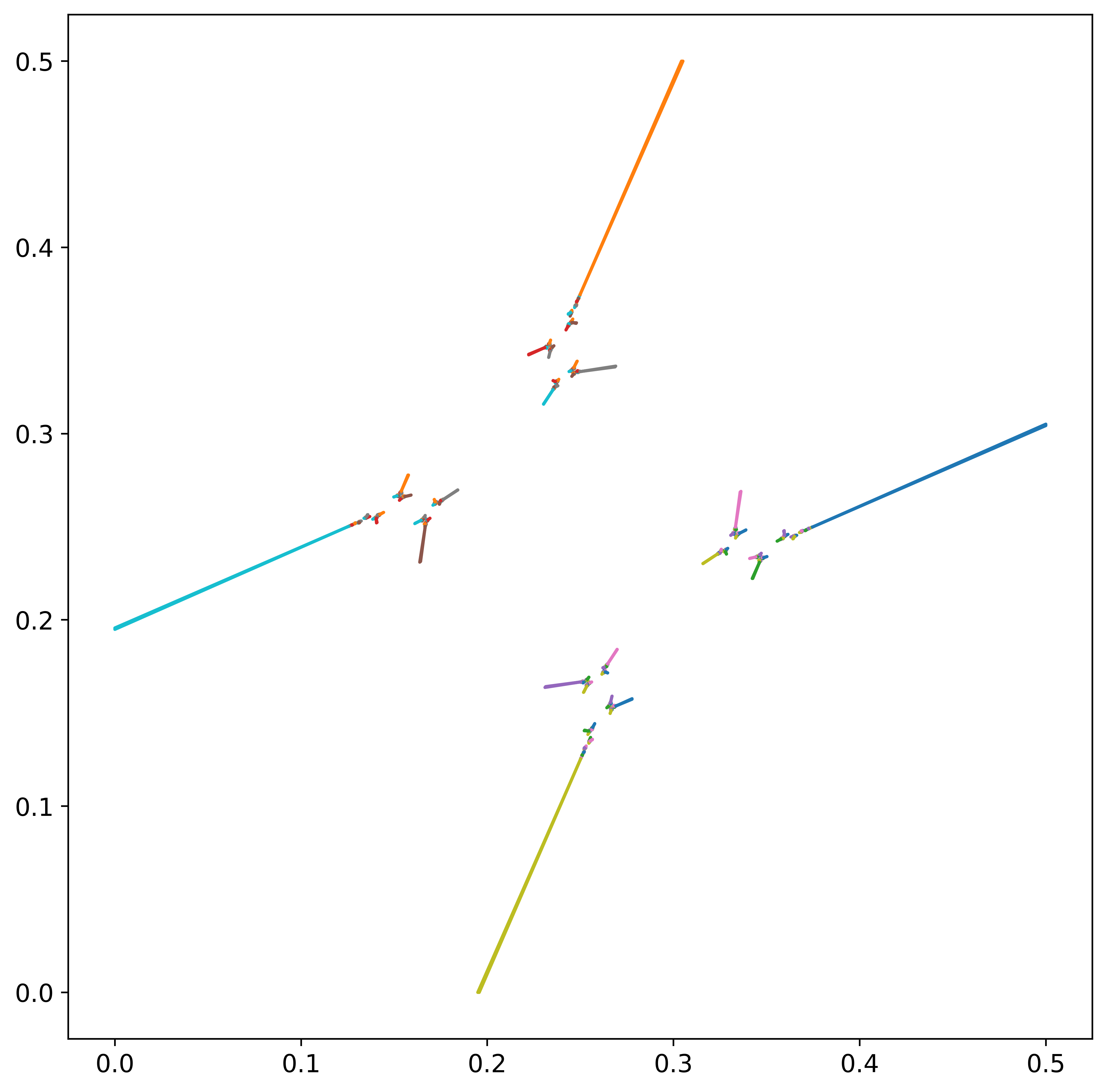}}}%
        \quad
        \subfloat[$\{\alpha,\beta\} = \{1,17\}$\label{fig:approx-1,17}]{{\includegraphics[width=.33\textwidth]{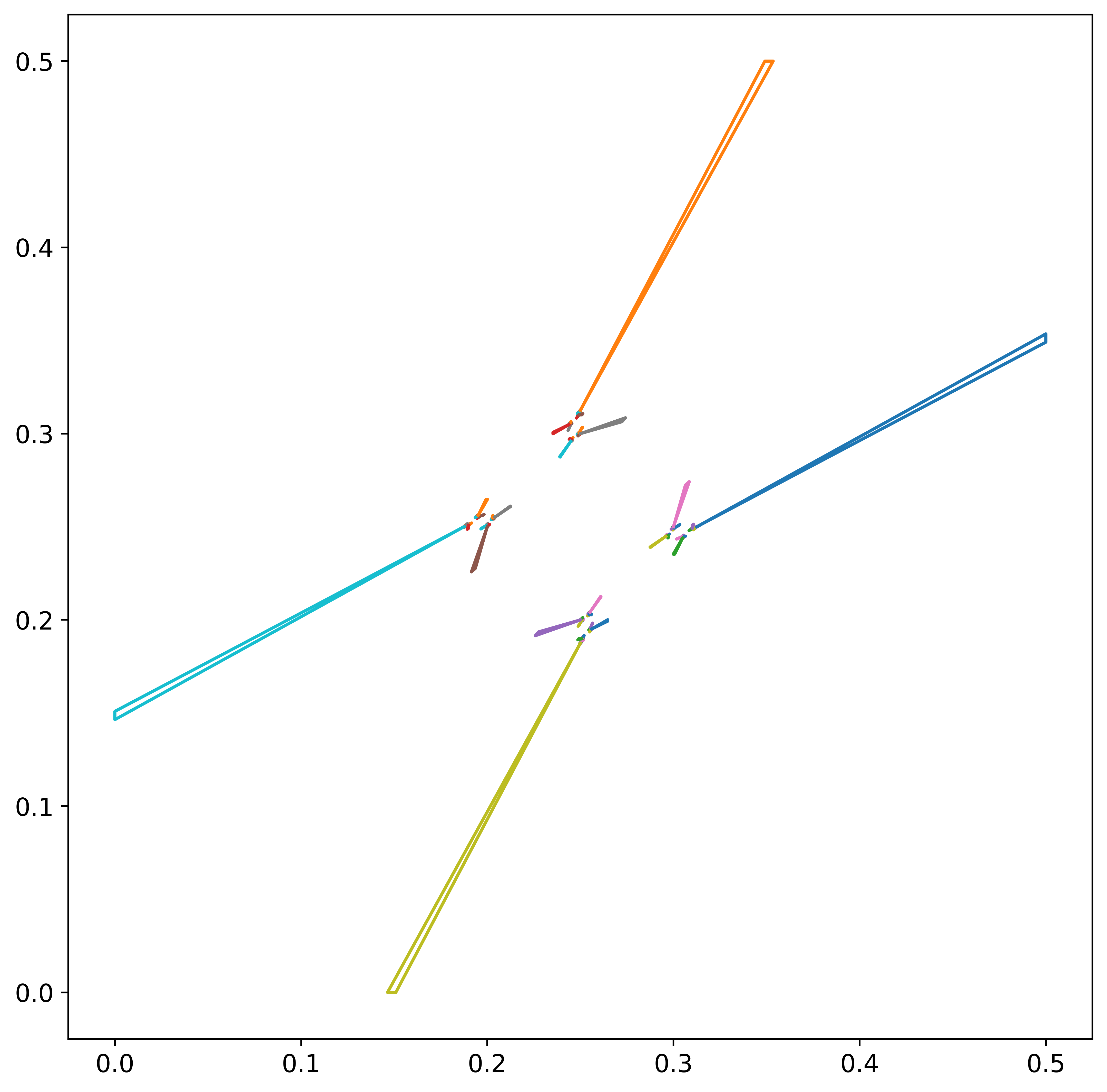}}}
        \caption{Images of the square $K$ by sequences of length $13$ of iterations of the two homographies $h_0$ and $h_1$.\label{fig:approx-other-alphabet-non-contracting}}
    \end{figure}
    
    On the other hand, when $\min\{\alpha,\beta\} \ge 3$, the homographies seem to contract event faster than for $\{1,3\}$. See for example Figures \ref{fig:approx-3,7} and \ref{fig:approx-9,11} where we had to left the traces of the first iterations since the contraction was happening really fast.

    \begin{figure}[!ht]%
        \centering
        \subfloat[\centering $\{\alpha,\beta\} = \{3,7\}$\label{fig:approx-3,7}]{{\includegraphics[width=.33\textwidth]{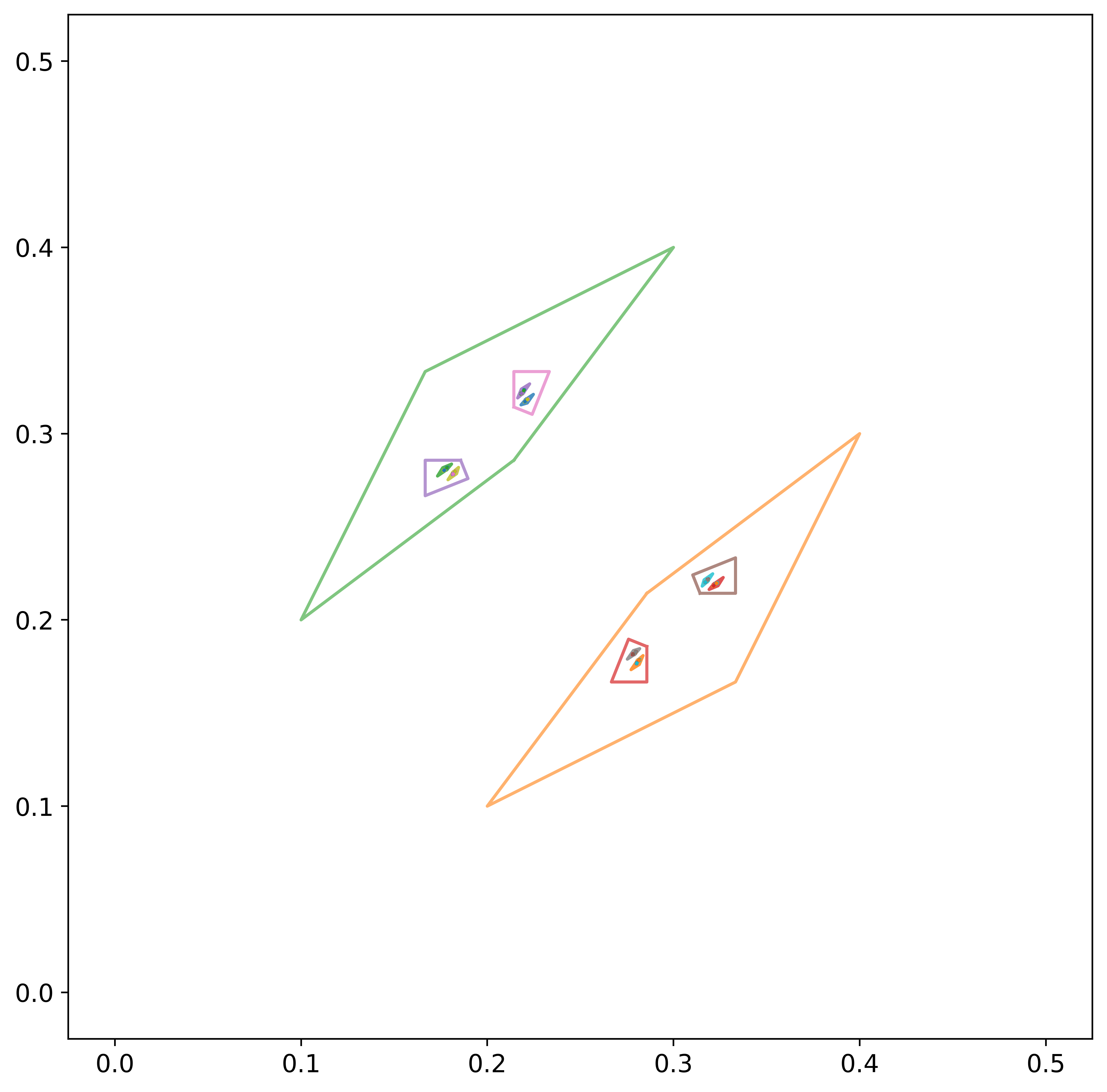}}}%
        \quad
        \subfloat[$\{\alpha,\beta\} = \{9,11\}$\label{fig:approx-9,11}]{{\includegraphics[width=.33\textwidth]{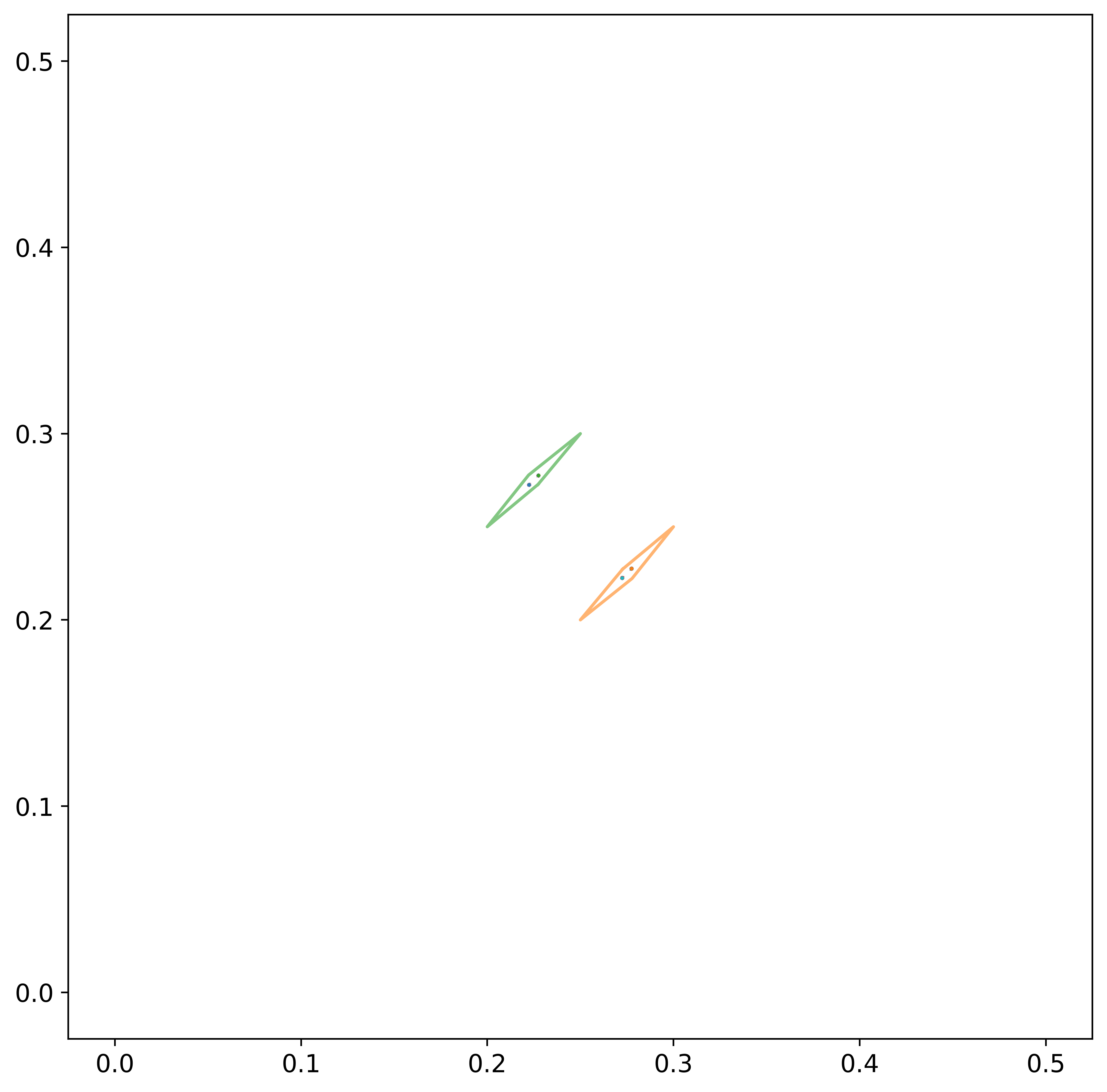}}}
        \caption{Images of the square $K$ by sequences of length up to $6$ of iterations of the two homographies $h_0$ and $h_1$.\label{fig:approx-other-alphabet-contracting}}
    \end{figure}
    
    \begin{problem} For which values of $\alpha$ and $\beta$ can we use a contracting property of the homographies $h_0$ and $h_1$ to prove the unicity of $a(\tau)$ and $b(\tau)$? Are the subshifts $(X_{\tau})_{\tau\in\{0,1\}^{\mathbb N}}$ of smooth bi-infinite sequences on the alphabet $\{\alpha, \beta\}$ uniquely ergodic for \emph{any} choice of odd integers $\alpha$ and $\beta$?
    \end{problem}

\subsection{Measure-preserving systems}

The analysis conducted in the paper provides a continuum of measure-preserving systems: $(X_\tau,\mu_\tau,S)_{\tau \in \{0,1\}^\NN}$. 
All these systems are ergodic, and have zero Kolmogorov-Sinai entropy: this latter fact follows from the variational principle, since the subshift $X$ of smooth sequences over $\{1,3\}$ has zero topological entropy (it has polynomial complexity, for details see the recent study on this topic~\cite{cassaigne2026}). But apart from this, it seems that none of their other measure-theoretic properties are obvious. 

\begin{problem} May the measure-preserving systems $(X_\tau,\mu_\tau,S)_{\tau \in \{0,1\}^\NN}$ have nontrivial measurable eigenfunctions? Can they be strongly mixing? Can they be of finite rank? Loosely Bernoulli? Can we describe the joinings between $(X_{\tau},\mu_{\tau},S)$ and $(X_{\tau'},\mu_{\tau'},S)$ for $\tau,\tau' \in\{0,1\}^{\mathbb N}$?
\end{problem}

\bibliography{smooth}

\end{document}